\documentclass[12pt]{article}

\usepackage{amssymb}
\usepackage{amsmath}
\usepackage{amsthm}

\usepackage[utf8]{inputenc}
\usepackage[english]{babel}
\usepackage{color} 
  
\usepackage[a4paper,margin=2.5cm]{geometry}
\theoremstyle{plain}
\newtheorem{theorem}{Theorem}[section]
\newtheorem{proposition}[theorem]{Proposition}
\newtheorem{lemma}[theorem]{Lemma}
\newtheorem{corollary}[theorem]{Corollary}
\newtheorem*{theorem2}{Theorem}

\theoremstyle{definition}
\newtheorem{definition}[theorem]{Definition}
\newtheorem{assumption}[theorem]{Assumption}
\newtheorem{remark}[theorem]{Remark}

\newtheorem*{assumption2}{Assumption}

\numberwithin{equation}{section}

\usepackage{enumerate}

\newcommand{\abs}[1]{\lvert{#1}\rvert}

\newcommand{\norm}[1]{\lVert{#1}\rVert}

\newcommand{\ip}[2]{\langle{#1},{#2}\rangle}

\DeclareMathOperator{\ran}{Ran}

\DeclareMathOperator*{\esup}{ess\ sup}
\DeclareMathOperator{\supp}{supp}
\DeclareMathOperator{\haus}{d_H}
\newcommand{\hausK}{\mathop{\textup{d}_{\textup{H},K}}}

\DeclareMathOperator{\dist}{dist}
\DeclareMathOperator{\rank}{rank}
\DeclareMathOperator{\trace}{Tr}

\newcommand{\vp}{\varphi}
\newcommand{\bR}{\mathbf{R}}
\newcommand{\bC}{\mathbf{C}}

\newcommand{\cF}{\mathcal{F}}

\newcommand{\cH}{\mathcal{H}}
\newcommand{\cD}{\mathcal{D}}

\newcommand{\cB}{\mathcal{B}}
\newcommand{\cK}{\mathcal{K}}
\newcommand{\cS}{\mathcal{S}}
\newcommand{\cI}{\mathcal{I}}

\newcommand{\bZ}{\mathbf{Z}}
\newcommand{\bT}{\mathbf{T}}
\newcommand{\bN}{\mathbf{N}}

\newcommand{\sfF}{\mathsf{F}}

\newcommand{\e}{\mathrm{e}}
\newcommand{\I}{\mathrm{i}}
\newcommand{\di}{\,\mathrm{d}}
\newcommand{\ident}{\ensuremath{I}}
\newcommand{\identh}{\ensuremath{I_h}}

\long\def\MSC#1\EndMSC{\def\arg{#1}\ifx\arg\empty\relax\else
	{\narrower\noindent%
		\small\textbf{2020 Mathematics Subject Classification:} #1} \fi}
\long\def\KEY#1\EndKEY{\def\arg{#1}\ifx\arg\empty\relax\else
	{\narrower\noindent%
		\small\textbf{Keywords:} #1\\}\fi}

\date{}

\begin{document}
\title{Norm resolvent convergence of \\discretized Fourier multipliers}
\author{
Horia Cornean\footnote{Department of Mathematical Sciences, Aalborg University, Skjernvej 4A, 
DK-9220 Aalborg \O{}, Denmark. Email: cornean@math.aau.dk and matarne@math.aau.dk},\;
Henrik Garde\footnote{Department of Mathematics, Aarhus University, Ny Munkegade 118, DK-8000 Aarhus C, Denmark. Email: garde@math.au.dk},\; and
Arne Jensen\footnotemark[1]
}

\maketitle

\begin{abstract}
	\noindent We prove norm estimates for the difference of resolvents of operators and their discrete counterparts, embedded into the continuum using biorthogonal Riesz sequences. The estimates are given in the operator norm for operators on square integrable functions, and depend explicitly on the mesh size for the discrete operators. The operators are a sum of a Fourier multiplier and a multiplicative potential. The Fourier multipliers include the  fractional Laplacian and the pseudo-relativistic free Hamiltonian. The potentials are real, bounded, and H\"older continuous. As a side-product, the Hausdorff distance between the spectra of the resolvents of the continuous and discrete operators decays with the same rate in the mesh size as for the norm resolvent estimates. The same result holds for the spectra of the original operators in a local Hausdorff distance.
\end{abstract}

\KEY
norm resolvent convergence, Fourier multiplier, lattice, Hausdorff distance.
\EndKEY

\MSC
47A10, 
47A58, 
42B15, 
(%
47A11, 
47B39
).
\EndMSC

\section{Introduction} \label{sec:intro}
We start with some notation and definitions. Let  $h\bZ^d$ be lattices with mesh size $h>0$. Let
\begin{equation*}
\cH_h=\ell^2(h\bZ^d)\quad\text{with the norm}\quad
\norm{u_h}_{\ell^2(h\bZ^d)}
=
\bigl(h^d\sum_{n\in \bZ^d}\abs{u_h(n)}^2\bigr)^{1/2}.
\end{equation*}
Note that we use the index $n\in\bZ^d$ in the notation for $u_h\in\cH_h$. The dependence on the lattice parameter is given by the subscript $h$.

Write $\cH = \widehat{\cH}=L^2(\bR^d)$ and let
 $\cF\colon \cH \to \widehat{\cH}$ be the Fourier transform given by
\begin{equation*}
(\cF f)(\xi)=(2\pi)^{-d/2}\int_{\bR^d}\e^{-\I x\cdot\xi}f(x)\di x, \quad \xi\in\bR^d,
\end{equation*}
with adjoint $\cF^{\ast}\colon \widehat{\cH}\to\cH$. Moreover, for $f\in \cH$ we denote its Fourier transform by $\widehat{f}$. Let $\bT^d_h=[-\pi/h,\pi/h]^d$ and $\widehat{\cH}_h=L^2(\bT^d_h)$. The discrete Fourier transform 
$\sfF_h\colon \cH_h\to\widehat{\cH}_h$ and its adjoint 
$\sfF_h^{\ast}\colon\widehat{\cH}_h\to\cH_h$ are given by
\begin{align*}
(\sfF_h u_h)(\xi)
&=\frac{h^d}{(2\pi)^{d/2}}\sum_{n\in\bZ^d}u_h(n)\e^{-\I hn\cdot\xi},\quad \xi\in\bT^d_h,\\
(\sfF_h^{\ast}g)(n)&=\frac{1}{(2\pi)^{d/2}}\int_{\bT^d_h}\e^{\I hn\cdot\xi}g(\xi)\di\xi,\quad  n\in\bZ^d.
\end{align*} 
Normalization is chosen to ensure that $\cF$ and $\sfF_h$ are unitary operators.
 
Let $\vp_0,\psi_0\in \cH$ and define
\begin{equation*}
\vp_{h,k}(x)=\vp_0((x-hk)/h),\quad \psi_{h,k}(x)=\psi_0((x-hk)/h),
\quad h>0, \quad k\in\bZ^d.
\end{equation*}
We assume that the sequences $\{\vp_{1,k}\}_{k\in\bZ^d}$ and
$\{\psi_{1,k}\}_{k\in\bZ^d}$ are biorthogonal Riesz sequences in $\cH$. See section~\ref{sec:biorth} for details.
Embedding operators $J_h\colon\cH_h\to\cH$ and 
$\widetilde{J}_h\colon\cH_h\to\cH$ are defined as
\begin{equation*}
J_hu_h=\sum_{k\in\bZ^d}u_h(k)\vp_{h,k},\quad
\widetilde{J}_hu_h=\sum_{k\in\bZ^d}u_h(k)\psi_{h,k},\quad u_h\in\cH_h.
\end{equation*}
The discretization operators $K_h\colon\cH\to\cH_h$ are given by $K_h=(\widetilde{J}_h)^{\ast}$. Explicitly
\begin{equation*}
(K_hf)(k)=\frac{1}{h^d}\ip{\psi_{h,k}}{f},\quad f\in\cH,\quad
k\in\bZ^d.
\end{equation*}
We have $K_hJ_h=\identh$, the identity on $\cH_h$, and $J_hK_h$ a projection in $\cH$ onto $J_h\cH_h$.

We consider Fourier multipliers on $\cH$. Several different classes of multipliers are considered, each characterized by a list of assumptions. One of the classes is described by the following list of assumptions:
\begin{assumption2}
Assume 
\begin{equation*}
\alpha>\tfrac12,\quad \beta>-\tfrac12,\quad\text{and}\quad
\alpha\leq 1+\beta<2\alpha \leq 3+\beta.
\end{equation*}
Let $G_0\in C^1(\bR^d)$ be a real-valued and non-negative function. Assume 
\begin{enumerate}[\ \ (1)]
\item $G_0(0)=0$,
\item there exist $c>0$ and $c_0>0$ such that $G_0(\xi)\geq c\abs{\xi}^{\alpha}$, for all $\xi\in\bR^d$ with $\abs{\xi}\geq c_0$,
\item $\abs{\nabla G_0(\xi)}\leq c \abs{\xi}^{\beta}$, 
for all $\xi\in\bR^d$,
\item $G_0(\xi_1,\xi_2,\ldots,\xi_d) = G_0(\abs{\xi_1},\abs{\xi_2},\ldots,\abs{\xi_d})$, for all $\xi\in\bR^d$.
\end{enumerate}

We define the discretization of $G_0$ in Fourier space by
\begin{equation*}
G_{0,h}(\xi)=G_0(\tfrac{2}{h}\sin(\tfrac{h}{2}\xi_1),
\tfrac{2}{h}\sin(\tfrac{h}{2}\xi_2),\ldots,
\tfrac{2}{h}\sin(\tfrac{h}{2}\xi_d)), \quad h>0, \quad \xi\in\bR^d.
\end{equation*}
\end{assumption2}

\noindent
Using the Fourier transforms we then define the Fourier multipliers
\begin{equation*}
H_0=\cF^{\ast}G_0(\cdot)\cF\quad\text{on $\cH$}
\end{equation*}
and
\begin{equation*}
H_{0,h}=\sfF_h^{\ast}G_{0,h}(\cdot)\sfF_h
\quad\text{on $\cH_h$}.
\end{equation*}
Here $G_0(\cdot)$ and $G_{0,h}(\cdot)$ denote the operators of multiplication by the function $G_0$ and the function $G_{0,h}$, respectively.

Let $V$ be a real-valued, bounded, and H\"{o}lder continuous function of order $\theta$. We discretize $V$ as $V_h$ by $V_h(k)=V(hk)$, $k\in\bZ^d$. Let $H=H_0+V$ on $\cH$ and $H_h=H_{0,h}+V_h$ on $\cH_h$. Under the assumption above one of the main results can be stated as follows. Here $\theta'$  is defined in \eqref{theta'} and is only present for $V\not\equiv 0$.
\begin{theorem2}
Let $K\subset\bC\setminus\bR$ be compact. Let
\begin{equation*}
\gamma=\min\{2\alpha-1,2\alpha-\beta-1,\theta'\}.
\end{equation*}
Then there exist $C>0$ and $h_0>0$ such that
\begin{equation*}
\norm{J_h(H_h-z\identh)^{-1}K_h-(H-z\ident)^{-1}}_{\cB(\cH)}
\leq Ch^{\gamma}
\end{equation*}
for all $z\in K$ and all $h$ with $0<h\leq h_0$.
\end{theorem2}
Thus we have norm resolvent convergence with an explicit rate of convergence, depending on the parameters in the assumptions. The complete results, also including two other classes of Fourier multipliers, are stated in Theorem~\ref{thm54}.

The norm resolvent estimates imply some spectral results. For example in Theorem~\ref{resolventest}, for $\mu>0$ large enough we get the estimate
\begin{equation*} 
\haus\bigl(\sigma((H_h+\mu\identh)^{-1}),\sigma((H+\mu\ident)^{-1})\bigr)\leq C h^{\gamma},
\end{equation*}
where $\haus$ denotes the Hausdorff distance. 
The same result holds for the spectra of the original operators in a local Hausdorff distance, see Theorem~\ref{localest}.
In Theorem~\ref{eigen-thm} we give results on the relation between discrete eigenvalues of $H$ and $H_h$.

The choice of the discretization $G_{0,h}$ 
is motivated by the functional
calculus for self-adjoint operators. To explain this, write $-\Delta$ for the Laplacian on $\cH$ and $-\Delta_h$ for the standard discrete Laplacian given by
\begin{equation*}
(-\Delta_hu_h)(k)=h^{-2}\sum_{i=1}^d\bigl(2u_h(k)-u_h(k+e_i)
-u_h(k-e_i)\bigr),\quad k\in \bZ^d,\quad u_h\in\cH_h.
\end{equation*}
Here $e_i$, $i=1,2,\ldots,d$, are the canonical basis vectors in $\bR^d$.

Let $G_0(\xi)=g_0(\abs{\xi}^2)$, where $g_0\colon[0,\infty)\to\bR$ is a given function. Then $H_0$ defined above equals the operator $g_0(-\Delta)$ obtained from the functional calculus for $-\Delta$, and $H_{0,h}$ defined above equals the operator $g_0(-\Delta_h)$ obtained from the functional calculus for $-\Delta_h$.

Our results cover some operators $H_0$ of interest in mathematical physics: 
\begin{itemize}
\item The fractional Laplacian $(-\Delta)^{s/2}$, $0<s<2$. 
The result for $0<s\leq1$ is covered by case (iii) in Theorem~\ref{thm54}.
For $1<s<2$ the above assumptions are satisfied with $\alpha=s$ and $\beta=s-1$. In both cases the rate is $h^s$.

\item The Laplacian $-\Delta$ and the bi-Laplacian $(-\Delta)^2$, and more generally $(-\Delta)^{s/2}$, $s\geq2$. The above assumptions are satisfied with $\alpha=(s+2)/2$ and $\beta=s-1$. The rate is $h^2$ for $s\geq2$.

\item The pseudo-relativistic free Hamiltonian $\sqrt{-\Delta+m^2}-m$, $m\geq0$. We obtain a rate of $h$, using $\alpha = 1$ and $\beta = 0$.
\end{itemize}
Moreover, for $V\not\equiv 0$ these rates are at most $h^{\theta'}$.

Concerning previous results,  Nakamura and Tadano~\cite{NT} were the  first authors to obtain norm resolvent convergence for the Laplacian plus a potential and its discretization, as far as we know. They use an orthonormal sequence to define embedding and discretization operators. A noteworthy difference, however, is that they consider a larger class of potentials $V$ that are allowed to be unbounded from above. The proof in \cite{NT} for unbounded $V$ cannot directly be adapted to the general class of $H_0$ in our setting. For specific choices of $G_0$ this might be possible, but it seems that no general statement covering the class of $H_0$ considered here can be obtained for a class of unbounded potentials. The underlying technical issue is that such a $V$ is not $H_0$-bounded, but in case of the Laplacian it is $H$-bounded; see e.g.\ the proofs of \cite[Lemmas~2.4 and 2.5]{NT} involving the specific structure of the Laplacian.

Many papers prove strong or weak convergence of solutions to equations involving the Laplacian or the fractional Laplacian, often in the context of non-linear equations. See for example the paper by Hong and Yang~\cite{HoYa19} and the references therein. 

Bambusi and Penati~\cite{BP} approximate a class of non-linear Schr\"{o}dinger equations in dimensions one and two by discrete non-linear Schr\"{o}dinger equations using a finite element method. With these results they prove existence of soliton-type solutions to the discretized equations based on known existence results for the continuous equations.

Rabinovich~\cite{R} proves results on the relation between the spectra of Schr\"{o}dinger operators on $h\bZ^d$ and on $\bR^d$. We comment on one of these results in Remark~\ref{rmk57}.

In~\cite{IJ} strong convergence results are obtained for solutions to discretized Schr\"{o}dinger equations on a number of different lattices. The results imply strong convergence of scattering operators.

We end this discussion by mentioning that the symmetry assumption which requires $G_0$ to depend on the absolute value of each $\xi_j$ excludes first order Fourier multipliers of Dirac type. This case, together with models involving constant magnetic fields, will be considered in a future work. 

The rest of this article is organized as follows. In section~\ref{sec:biorth} we define the operators $J_h$ and $K_h$ using biorthogonal Riesz sequences with certain support properties. We also give constructive examples of such biorthogonal Riesz sequences satisfying the assumptions required by $J_h$ and $K_h$. In section~\ref{sec:fouriermult} we introduce the assumptions on the Fourier multipliers, and give estimates of differences of resolvents in Propositions~\ref{prop-77}, \ref{prop126}, and \ref{prop410}. Finally in section~\ref{sec:withpotential} we also add a potential $V$, and collect the general norm resolvent estimates in Theorem~\ref{thm54}. Section~\ref{sec:spectral} is dedicated to spectral estimates.

\section{Biorthogonal systems} \label{sec:biorth}
We use biorthogonal Riesz sequences to define the embedding operators $J_h$ and the discretization operators $K_h$.

\begin{definition}
Let $\{u_k\}_{k\in\bZ^d}$ and $\{v_k\}_{k\in\bZ^d}$ be two sequences in a Hilbert space $\cK$. They are said to be biorthogonal, if we have
\begin{equation}\label{biorth}
\ip{u_k}{v_n}=\delta_{k,n}, \quad k,n\in\bZ^d,
\end{equation}
where $\delta_{k,n}$ is Kronecker's delta.
\end{definition}
We recall the definition of a Riesz sequence (see e.g.~\cite[Section 3.6]{OC}). 
\begin{definition}\label{def-62}
A sequence $\{u_k\}_{k\in\bZ^d}$ in a Hilbert space $\cK$ is called a Riesz sequence if there exist $A>0$ and $B>0$ such that
\begin{equation}\label{riesz}
A\sum_{k\in\bZ^d}\abs{c_k}^2\leq\norm{\sum_{k\in\bZ^d}c_ku_k}^2
\leq B\sum_{k\in\bZ^d}\abs{c_k}^2
\end{equation}
for all $\{c_k\}_{k\in\bZ^d}\in\ell^2(\bZ^d)$.
\end{definition}

Let $\vp_0,\psi_0\in \cH$. Define 
\begin{equation}\label{scaling}
\vp_{h,k}(x)=\vp_0((x-hk)/h),\quad \psi_{h,k}(x)=\psi_0((x-hk)/h),
\quad h>0, \quad k\in\bZ^d.
\end{equation}

\begin{assumption}\label{assume-63}
Assume that the sequences $\{\vp_{1,k}\}_{k\in\bZ^d}$ and
$\{\psi_{1,k}\}_{k\in\bZ^d}$ are biorthogonal Riesz sequences in $\cH$.
\end{assumption}

Note that \eqref{biorth} and \eqref{scaling} imply that
$\{h^{-d/2}\vp_{h,k}\}_{k\in\bZ^d}$ and
$\{h^{-d/2}\psi_{h,k}\}_{k\in\bZ^d}$ are biorthogonal sequences.
They are also Riesz sequences. After a change of variables we get from \eqref{riesz} the estimates
\begin{align}
A_{\vp}h^d\sum_{k\in\bZ^d}\abs{c_k}^2&\leq
\norm{\sum_{k\in\bZ^d}c_k\vp_{h,k}}^2
\leq B_{\vp}h^d\sum_{k\in\bZ^d}\abs{c_k}^2,\label{riesz-h-1}
\\
A_{\psi}h^d\sum_{k\in\bZ^d}\abs{c_k}^2&\leq
\norm{\sum_{k\in\bZ^d}c_k\psi_{h,k}}^2
\leq B_{\psi}h^d\sum_{k\in\bZ^d}\abs{c_k}^2. \notag
\end{align}

The embedding operators $J_h\colon\cH_h\to\cH$ are defined by
\begin{equation*}
J_hu_h=\sum_{k\in\bZ^d}u_h(k)\vp_{h,k},\quad u_h\in\cH_h.
\end{equation*}
The estimate \eqref{riesz-h-1} implies that $J_h$ is a well defined bounded operator from $\cH_h$ to $\cH$. Furthermore, the estimate \eqref{riesz-h-1} implies
\begin{equation}
\norm{J_h}_{\cB(\cH_h,\cH)}\leq \sqrt{B_{\vp}}\quad \text{for all $h>0$}.\label{J-norm}
\end{equation}

Let
\begin{equation*}
\widetilde{J}_hu_h=\sum_{k\in\bZ^d}u_h(k)\psi_{h,k},\quad u_h\in\cH_h.
\end{equation*}
It is also a well defined bounded operator from $\cH_h$ to $\cH$ and satisfies $\norm{\widetilde{J_h}}_{\cB(\cH_h,\cH)}\leq \sqrt{B_{\psi}}$ for all $h>0$.

The discretization operators are defined as $K_h=(\widetilde{J}_h)^{\ast}$. They are bounded operators from $\cH$ to $\cH_h$, and are given by
\begin{equation*}
(K_hf)(k)=\frac{1}{h^d}\ip{\psi_{h,k}}{f}, \quad k\in\bZ^d.
\end{equation*}
They satisfy 
\begin{equation}
\norm{K_h}_{\cB(\cH,\cH_h)}\leq \sqrt{B_{\psi}}\quad \text{for all $h>0$}.
\label{K-norm}
\end{equation}
Due to biorthogonality we have that
\begin{equation}\label{KhJh}
K_hJ_h=\identh,
\end{equation}
where $\identh$ is the identity on $\cH_h$. Thus $J_h$ is injective and $K_h$ is surjective.

We also have 
\begin{equation*}
J_hK_hf = \sum_{k\in\bZ^d}\frac{1}{h^d}\ip{\psi_{h,k}}{f}\vp_{h,k}
\end{equation*} 
and the biorthogonality condition (or \eqref{KhJh}) implies that $J_hK_h$ is a projection, although not necessarily an orthogonal projection.

\begin{lemma} \label{lemma:proj}
	Let the sequences $\{\vp_{1,k}\}_{k\in\bZ^d}$ and
	$\{\psi_{1,k}\}_{k\in\bZ^d}$ satisfy Assumption~{\rm\ref{assume-63}}. Then $J_hK_h\neq I$ for all $h>0$.
\end{lemma}
\begin{proof}
We start with the case $d=1$ and $h=1$. We will find $f\in\cH$ such that $f\neq0$ and $K_hf=0$. We define $f$ via the Fourier transform by
	\begin{align*}
		\widehat{f}(\xi+2\pi(2k))&= \overline{\widehat{\psi}_0(
		\xi+2\pi(2k+1))},\\
		\widehat{f}(\xi+2\pi(2k+1))&= -\overline{\widehat{\psi}_0(
		\xi+2\pi(2k))},
	\end{align*}
	for $\xi\in[-\pi,\pi)$ and $k\in\bZ$.
	This $f$ then satisfies
	\begin{equation*}
		\sum_{n\in\bZ}\widehat{f}(\xi+2\pi n)\overline{\widehat{\psi}_0(
		\xi+2\pi n)}=0
	\end{equation*}
	for a.e.\ $\xi\in\bR$. 
	
	Then
	\begin{equation*}
		\ip{\psi_{1,j}}{f}=\int_{\bR}\widehat{f}(\xi)\overline{\widehat{\psi}_0(\xi)}\e^{\I j\xi}\di\xi
		=\int_{-\pi}^{\pi}\Bigl(
		\sum_{n\in\bZ}\widehat{f}(\xi+2\pi n)\overline{\widehat{\psi}_0(
		\xi+2\pi n)}\Bigr)\e^{\I j\xi}\di\xi=0
	\end{equation*}
	for all $j\in\bZ$, implying that $K_1f=0$. By construction we have $f\neq0$. For general $h>0$ the result follows by a change of variables. The case $d>1$ follows by
applying the above argument to the variable $\xi_1$.
\end{proof}

We state the following elementary result from operator theory. For the reader's convenience we include a short proof.
The setting is as follows.
Let $\cK_1$ and $\cK_2$ be Hilbert spaces. Let $S\in\cB(\cK_1,\cK_2)$ and $T\in\cB(\cK_2,\cK_1)$ and assume that $TS=I_{\cK_1}$, the identity on $\cK_1$. Then $P=ST$ is a projection onto the closed subspace $\ran S$.

\begin{lemma}\label{lemma-spect2}
Let $A\in\cB(\cK_1)$. Then the following results hold.
\begin{itemize}
\item[\rm(i)] Assume $P=I_{\cK_2}$. Then
\begin{equation*}
\sigma(SAT)=\sigma(A).
\end{equation*}
\item[\rm(ii)] Assume $P\neq I_{\cK_2}$. Then
\begin{equation*}
\sigma(SAT)=\sigma(A)\cup\{0\}.
\end{equation*}
\end{itemize}
\end{lemma}
\begin{proof}
In case (i) the result is obvious, since $T=S^{-1}$. In case (ii)  $SAT$ restricted to $\ran S$ is a bounded operator in $\cB(\ran S)$ and the first case can be applied here. Since $\cK_2=P\cK_2\oplus(I_{\cK_2}-P)\cK_2$ and $(I_{\cK_2}-P)\cK_2$ is nontrivial, the result follows.
\end{proof}

We apply this lemma to $J_h$ and $K_h$ as defined above to get the following result, which relates the spectrum of an operator on the discrete space $\cH_h$ to the spectrum of its embedding in $\cH$.
\begin{proposition}\label{prop-spect}
Let $F_h\in\cB(\cH_h)$. Then
\begin{equation*}
\sigma(J_h F_h K_h)=\sigma(F_h)\cup	\{0\} \quad \text{for all } h>0.
\end{equation*}
\end{proposition}
\begin{proof}
This result follows from Lemmas~\ref{lemma:proj} and~\ref{lemma-spect2}.
\end{proof}

The following result is well-known from wavelet theory, see e.g.~\cite{meyer}.
\begin{lemma}\label{lemma71}
Let $\varphi,\psi\in \cH$. Then the two systems
\begin{equation*} 
\{\varphi(\cdot-k)\}_{k\in\bZ^d}\quad\text{and}\quad 
\{\psi(\cdot-k)\}_{k\in\bZ^d}
\end{equation*}
are biorthogonal if and only if
\begin{equation}\label{BIO-cond}
\sum_{k\in\bZ^d}\overline{\widehat{\varphi}(\xi-2\pi k)}\widehat{\psi}(\xi-2\pi k)=(2\pi)^{-d}, \quad
\text{a.e. $\xi\in\bR^d$}.
\end{equation}
\end{lemma}

Recall that $\widehat{\cH}=\cF(\cH)=L^2(\bR^d)$ and 
$\widehat{\cH}_h=\sfF_h(\cH_h)=L^2(\bT_h^d)$. We start by transferring the maps $J_h$ and $K_h$ to the Fourier image spaces.
First note that
\begin{equation*}
(\cF \vp_{h,k})(\xi)=h^d\e^{-\I hk\cdot\xi}\widehat{\vp}_0(h\xi),\quad \xi\in\bR^d.
\end{equation*}
Thus for $u_h\in \cH_h$ we have
\begin{equation}\label{FJh}
(\cF J_h u_h)(\xi)=h^d\widehat{\vp}_0(h\xi)
\sum_{k\in\bZ^d}u_h(k)\e^{-\I hk\cdot\xi}
=(2\pi)^{d/2}\widehat{\vp}_0(h\xi)(\widetilde{\sfF_hu_h})(\xi), \quad \xi\in\bR^d,
\end{equation}
where $\widetilde{\sfF_hu_h}$ denotes the periodic extension of
$\sfF_hu_h$ from $\bT_h^d$ to $\bR^d$.

For $K_h$ we compute as follows. Let $f\in\cH$.
\begin{align*}
(K_hf)(k)&=\frac{1}{h^d}\ip{\psi_{h,k}}{f}\\
&=\frac{1}{h^d}\ip{\widehat{\psi}_{h,k}}{\widehat{f}}\\
&=\ip{\e^{-\I hk\cdot(\cdot)}\widehat{\psi}_0(h(\cdot))}{\widehat{f}}\\
&=\sum_{j\in\bZ^d}\int_{\bT^d_h-\frac{2\pi}{h}j}
\e^{\I hk\cdot\xi}\overline{\widehat{\psi}_0(h\xi)}\widehat{f}(\xi)\di\xi\\
&=\int_{\bT^d_h}
\e^{\I hk\cdot\xi}\sum_{j\in\bZ^d}\overline{\widehat{\psi}_0(h\xi+2\pi j)}\widehat{f}(\xi+\tfrac{2\pi}{h}j)\di\xi\\
&=(2\pi)^{d/2}(\sfF_h^{\ast}\Psi_h \widehat{f})(k), \quad k\in\bZ^d,
\end{align*}
where
\begin{equation*}
(\Psi_h \widehat{f})(\xi)=\sum_{j\in\bZ^d}\overline{\widehat{\psi}_0(h\xi+2\pi j)}\widehat{f}(\xi+\tfrac{2\pi}{h}j), \quad \xi\in \bR^d,
\end{equation*}
is a $(2\pi/h)\bZ^d$-periodic function. Thus for $g\in\widehat{\cH}$,
\begin{equation}\label{FKh}
\sfF_hK_h\cF^{\ast}g=(2\pi)^{d/2}\Psi_h g.
\end{equation}
Combining the results \eqref{FJh} and \eqref{FKh} we finally get
\begin{equation}\label{FJK}
(\cF J_h K_h \cF^{\ast}g)(\xi)=(2\pi)^{d}\widehat{\vp}_0(h\xi)
\sum_{j\in\bZ^d}\overline{\widehat{\psi}_0(h\xi+2\pi j)}g(\xi+\tfrac{2\pi}{h}j), \quad \xi\in\bR^d.
\end{equation}
This result is an extension of \cite[Lemma 2.1]{NT} to biorthogonal systems.

We introduce the following assumptions.
\begin{assumption}\label{assume-72}
Let $\widehat{\vp}_0,\widehat{\psi}_0\in \cH$ be essentially bounded and satisfy Assumption~{\rm\ref{assume-63}}. Assume further that there exists $c_0>0$ such that
\begin{equation}\label{supp-cond-phi}
\supp(\widehat{\vp}_0)\subseteq [-3\pi/2,3\pi/2]^d
\quad\text{and}\quad \abs{\widehat{\vp}_0(\xi)}\geq c_0,\quad
\xi\in[-\pi/2,\pi/2]^d,
\end{equation}
and
\begin{equation}\label{supp-cond-psi}
\supp(\widehat{\psi}_0)\subseteq [-3\pi/2,3\pi/2]^d
\quad\text{and}\quad \abs{\widehat{\psi}_0(\xi)}\geq c_0,\quad
\xi\in[-\pi/2,\pi/2]^d.
\end{equation}
\end{assumption}

It is not obvious that there exist $\vp_0\neq\psi_0$ such that Assumption~\ref{assume-72} holds. We construct a class of examples in  the next subsection.

\subsection{Examples of biorthogonal Riesz sequences} \label{sec:biorth_example}
For $j=1,2$ take $u_j\in C_0^{\infty}(\bR^d)$, such that $0\leq u_j(\xi)\leq1$,
\begin{equation*}
u_j(\xi)=1,\quad \xi\in[-\pi,\pi]^d,\quad
\supp u_j\subseteq [-3\pi/2,3\pi/2]^d.
\end{equation*}
Define   
\begin{equation*}
v(\xi)=\sum_{k\in\bZ^d}{u_1(\xi-2\pi k)}u_2(\xi-2\pi k), \quad \xi\in\bR^d.
\end{equation*}
Then $v(\xi-2\pi n)=v(\xi)$ for all $n\in\bZ^d$, and $v(\xi)\geq1$ for all $\xi\in\bR^d$. 
It follows that
\begin{equation*}
\sum_{k\in\bZ^d}\frac{{u_1(\xi-2\pi k)}u_2(\xi-2\pi k)}{v(\xi)}=1, \quad \xi\in\bR^d.
\end{equation*}

Define 
\begin{align}
\widehat{\vp}_0&
=(2\pi)^{-d/2}\frac{u_1}{v^{1/2}},\label{def-phi}\\
\widehat{\psi}_0&
=(2\pi)^{-d/2}\frac{u_2}{v^{1/2}}.\label{def-psi}
\end{align}
From the construction and Lemma~\ref{lemma71}, the  
sequences 
$\{\vp_{1,k}\}_{k\in\bZ^d}$ and $\{\psi_{1,k}\}_{k\in\bZ^d}$ obtained from \eqref{def-phi} and \eqref{def-psi} are biorthogonal.

\begin{remark} Variations in the construction in \eqref{def-phi} and \eqref{def-psi} are possible.
	\begin{enumerate}[(i)]
		\item Fix $\delta\in[0,1]$ and define
		\begin{equation*}
		\widehat{\vp}_{0}
		=(2\pi)^{-d\delta}\frac{u_1}{v^{\delta}},\quad
		\widehat{\psi}_{0}
		=(2\pi)^{-d(1-\delta)}\frac{u_2}{v^{1-\delta}}.
		\end{equation*}
		In particular, one can get either $\widehat{\vp}_{0}=u_1$ or $\widehat{\psi}_{0}=u_2$. 
		\item We can replace the condition $u_j\in C_0^{\infty}(\bR^d)$, $j=1,2$, by a finite regularity condition 
		$u_j\in C_0^{k}(\bR^d)$, $j=1,2$, $k\geq0$. Although, for the later Assumption~\ref{assume-beta} one may wish to ensure some regularity of $\widehat{\psi}_0$.
	\end{enumerate}
\end{remark}

The next step is to verify that both  sequences 
 are Riesz sequences.
We recall a definition, see~\cite[Definition~3.2.2]{OC}. 
\begin{definition}
A sequence $\{u_k\}_{k\in\bZ^d}$ in a Hilbert space $\cK$ is called a Bessel sequence, if there exists $C>0$ such that
\begin{equation*} 
\sum_{k\in\bZ^d}\abs{\ip{u_k}{u}}^2\leq C\norm{u}^2
\end{equation*}
for all $u\in\cK$.
\end{definition}

\begin{lemma}
The sequences $\{\vp_{1,k}\}_{k\in\bZ^d}$ and $\{\psi_{1,k}\}_{k\in\bZ^d}$, obtained from \eqref{def-phi} and \eqref{def-psi}, respectively, are both Bessel sequences.
\end{lemma}
\begin{proof}
It suffices to give the proof for $\{\vp_{1,k}\}_{k\in\bZ^d}$.
We have
\begin{align*}
\ip{\vp_{1,k}}{u}&=\int_{\bR^d}\  \overline{{\vp_0(x-k)}}u(x)\di x
=\int_{\bR^d}\overline{\widehat{\vp}_0(\xi)}
\widehat{u}(\xi)\e^{\I k\cdot\xi}\di\xi\\
&=\int_{[-2\pi,2\pi]^d}
\overline{\widehat{\vp}_0(\xi)}
\widehat{u}(\xi)\e^{\I k\cdot\xi}\di\xi.
\end{align*}
Thus $\ip{\vp_{1,k}}{u}$, $k\in\bZ^d$, are some of the Fourier coefficients of the function $\overline{\widehat{\vp}_0}\widehat{u}$
on $[-2\pi,2\pi]^d$. Hence
\begin{equation*}
\sum_{k\in\bZ^d}\abs{\ip{\vp_{1,k}}{u}}^2\leq
C \norm{\overline{\widehat{\vp}_0}\widehat{u}}^2_{L^2([-2\pi,2\pi]^d)}
\leq C\norm{u}^2_{L^2(\bR^d)}.
\end{equation*}
Here we have used boundedness of $\widehat{\vp}_0$ and the support property \eqref{supp-cond-phi}.
\end{proof}

The upper bound in Definition~\ref{def-62} is obtained in the following lemma. The proof is a variant of part of the proof of \cite[Theorem~3.2.3]{OC}.
\begin{lemma}\label{lemma-79}
The sequences $\{\vp_{1,k}\}_{k\in\bZ^d}$ and $\{\psi_{1,k}\}_{k\in\bZ^d}$, obtained from \eqref{def-phi} and \eqref{def-psi}, respectively, both satisfy the right hand inequality in \eqref{riesz}.
\end{lemma}
\begin{proof}
It suffices to consider $\{\vp_{1,k}\}_{k\in\bZ^d}$. Let $T\colon\ell^2(\bZ^d)\to\cH$ be given by
\begin{equation*}
T(\{c_k\}_{k\in\bZ^d})=\sum_{k\in\bZ^d}c_k\vp_{1,k}.
\end{equation*}
As a first step we show that the series is convergent. For $k\in\bZ^d$ we use the standard notation $\abs{k}=\sum_{i=1}^d\abs{k_i}$. Let $m,n\in\bZ$ and assume $-1\leq m<n$.
\begin{align*}
\norm{\sum_{\abs{k}\leq m}c_k\vp_{1,k}-
\sum_{\abs{k}\leq n}c_k\vp_{1,k}}&=
\norm{\sum_{m<\abs{k}\leq n}c_k\vp_{1,k}}\\
&=\sup_{\norm{u}=1}
\abs{\ip{u}{\sum_{m<\abs{k}\leq n}c_k\vp_{1,k}}}\\
&\leq \sup_{\norm{u}=1}\sum_{m<\abs{k}\leq n}\abs{\ip{u}{c_k\vp_{1,k}}}\\
&\leq \Bigl(\sum_{m<\abs{k}\leq n}
\abs{c_k}^2
\Bigr)^{1/2}
\Bigl(\sup_{\norm{u}=1}\sum_{m<\abs{k}\leq n}
\abs{\ip{u}{\vp_{1,k}}}^2
\Bigr)^{1/2}\\
&\leq \sqrt{B_{\vp}} \Bigl(\sum_{m<\abs{k}\leq n}
\abs{c_k}^2
\Bigr)^{1/2},
\end{align*}
since $\{\vp_{1,k}\}_{k\in\bZ^d}$ is a Bessel sequence, with a constant denoted by $B_{\vp}$. Thus the sequence $\{\sum_{\abs{k}\leq n}c_k\vp_{1,k}\}_{n\in\bN}$ is convergent in $\cH$. The upper Riesz bound follows from the computation above, if we take $m=-1$ and let $n\to\infty$.
\end{proof}

A sequence $\{c_k\}_{k\in\bZ^d}$ is said to be finite if at most finitely many $c_k$ are non-zero. 
\begin{lemma}\label{lemma-710}
For the sequences $\{\vp_{1,k}\}_{k\in\bZ^d}$ and $\{\psi_{1,k}\}_{k\in\bZ^d}$, obtained from \eqref{def-phi} and \eqref{def-psi}, respectively, there exist $A_{\vp}>0$ and $A_{\psi}>0$ such that
\begin{align}
\norm{\sum_{k\in\bZ^d} c_k\vp_{1,k}}^2&\geq A_{\vp}\sum_{k\in\bZ^d}\abs{c_k}^2,
\label{lower-phi}\\
\intertext{and}
\norm{\sum_{k\in\bZ^d} c_k\psi_{1,k}}^2&\geq A_{\psi}\sum_{k\in\bZ^d}\abs{c_k}^2,
\label{lower-psi}
\end{align}
for all finite sequences $\{c_k\}_{k\in\bZ^d}$.
\end{lemma} 
\begin{proof}
Let $u=\sum_{j\in\bZ^d}c_j\vp_{1,j}$ for a finite sequence $\{c_j\}_{j\in\bZ^d}$. Biorthogonality of $\{\vp_{1,k}\}_{k\in\bZ^d}$ and $\{\psi_{1,k}\}_{k\in\bZ^d}$ implies that $c_k=\ip{\psi_{1,k}}{u}$. Since $\{\psi_{1,k}\}_{k\in\bZ^d}$ is a Bessel sequence (with bound denoted by $B_{\psi}$), we have
\begin{equation*}
\sum_{k\in\bZ^d}\abs{c_k}^2=\sum_{k\in\bZ^d}\abs{\ip{\psi_{1,k}}{u}}^2
\leq B_{\psi}\norm{u}^2=B_{\psi}\norm{\sum_{k\in\bZ^d}c_k\vp_{1,k}}^2,
\end{equation*}
which yields \eqref{lower-phi} with $A_{\vp}=B_{\psi}^{-1}$. Exchanging the roles of $\vp_{1,k}$ and $\psi_{1,k}$ we get \eqref{lower-psi} with $A_{\psi}=B_{\vp}^{-1}$.
\end{proof}
 
Combining the results in Lemmas~\ref{lemma-79} and~\ref{lemma-710} we get:
\begin{proposition}
The sequences $\{\vp_{1,k}\}_{k\in\bZ^d}$ and $\{\psi_{1,k}\}_{k\in\bZ^d}$, obtained from \eqref{def-phi} and \eqref{def-psi}, respectively, are biorthogonal Riesz sequences.
\end{proposition}
Thus we have constructed a pair of sequences satisfying Assumption~\ref{assume-72}.

\section{Estimates for Fourier multipliers} \label{sec:fouriermult}

In this section we give norm resolvent estimates related to three different classes of Fourier multipliers. The results are collected in Propositions~\ref{prop-77}, \ref{prop126}, and \ref{prop410}.

We start by introducing the following assumption.
\begin{assumption}\label{def121}
Assume 
\begin{equation*} 
\alpha>\tfrac12,\quad \beta>-\tfrac12,\quad\text{and}\quad
\alpha\leq 1+\beta<2\alpha \leq 3+\beta.
\end{equation*}
Let $G_0\in C^1(\bR^d)$ be a real-valued and non-negative function. Assume 
\begin{enumerate}[\ \ (1)]
\item $G_0(0)=0$,
\item there exist $c>0$ and $c_0>0$ such that $G_0(\xi)\geq c\abs{\xi}^{\alpha}$, for all $\xi\in\bR^d$ with $\abs{\xi}\geq c_0$,
\item $\abs{\nabla G_0(\xi)}\leq c \abs{\xi}^{\beta}$, 
for all $\xi\in\bR^d$,
\item $G_0(\xi_1,\xi_2,\ldots,\xi_d) = G_0(\abs{\xi_1},\abs{\xi_2},\ldots,\abs{\xi_d})$, for all $\xi\in\bR^d$.
\end{enumerate}
\end{assumption}
\begin{remark} 
	The condition $\alpha\leq 1+\beta$ is necessary, since due to Assumption~\ref{def121}(1)
	\begin{equation*}
	G_0(\xi)=\int_0^1\xi\cdot\nabla G_0(t\xi)\di t.
	\end{equation*}
	The condition $1+\beta<2\alpha$ gives control over oscillations in $G_0$. Finally, the condition $2\alpha\leq 3+\beta$ is 
tied to the accuracy for the choice of discretization \eqref{def-G0h}. For details on this last part, see the proof of Lemma~\ref{lemma:diffresG} and Remark~\ref{remark:maxorder}.
\end{remark}

We define the discretized multiplier as
\begin{equation}\label{def-G0h}
G_{0,h}(\xi)=G_0(\tfrac{2}{h}\sin(\tfrac{h}{2}\xi_1),
\tfrac{2}{h}\sin(\tfrac{h}{2}\xi_2),\ldots,
\tfrac{2}{h}\sin(\tfrac{h}{2}\xi_d)), \quad h>0, \quad \xi\in\bR^d.
\end{equation}
Due to Assumption~\ref{def121}(4) the function $G_{0,h}(\xi)$ is $(2\pi/h)\bZ^d$-periodic. Using the Fourier transforms we then define
\begin{equation}\label{def-H0}
H_0=\cF^{\ast}G_0(\cdot)\cF\quad\text{on $\cH=L^2(\bR^d)$}
\end{equation}
and
\begin{equation}\label{def-H0h}
H_{0,h}=\sfF_h^{\ast}G_{0,h}(\cdot)\sfF_h
\quad\text{on $\cH_h=\ell^2(h\bZ^d)$}.
\end{equation}

Take $\vp_0,\psi_0\in \cH$ such that Assumption~\ref{assume-72} is satisfied. For $j\in\bZ^d$ the support conditions \eqref{supp-cond-phi} and \eqref{supp-cond-psi} imply
\begin{equation}\label{supp-intersect}
\supp(\widehat{\vp}_0)\cap
\supp(\widehat{\psi}_0(\cdot-2\pi j))=\emptyset,\quad
\text{if $\abs{j}>1$}.
\end{equation}
Here $\abs{j}=\sum_{i=1}^d\abs{j_i}$.

\begin{lemma}\label{lemma-74}
Let $\vp_0$ and $\psi_0$ satisfy Assumption~{\rm\ref{assume-72}}. 
Let $G_0$ satisfy Assumption~{\rm\ref{def121}} and let $H_0$ be defined by \eqref{def-H0}. For each compact set $K\subset\bC\setminus[0,\infty)$,
there exist $C>0$ and $h_0>0$ such that
\begin{equation}\label{eq710}
\norm{(J_hK_h-\ident)(H_0-z\ident)^{-1}}_{\cB(\cH)}\leq C h^{\alpha},
\end{equation}
for all $z\in K$ and all $0<h\leq h_0$.
\end{lemma}
\begin{proof}
It suffices to prove the estimate for $K=\{-1\}$, since
 $(H_0+\ident)(H_0-z\ident)^{-1}$ is uniformly bounded in norm for $z\in K$.
Let $g\in\cS(\bR^d)$, the Schwartz space.
By the computations leading to \eqref{FJK}, we have
\begin{align*}
\bigl(\cF&(J_hK_h-\ident)(H_0+\ident)^{-1}\cF^{\ast}g\bigr)(\xi)\\
&=
(2\pi)^d\widehat{\vp}_0(h\xi)
\sum_{j\in\bZ^d}\overline{\widehat{\psi}_0(h\xi+2\pi j)}(G_0(\xi+\tfrac{2\pi}{h}j)+1)^{-1}g(\xi+\tfrac{2\pi}{h}j)\\
&\quad - (G_0(\xi)+1)^{-1}g(\xi), \quad \xi\in\bR^d.
\end{align*}
For $h\xi\in[-\pi/2,\pi/2]^d$ the $j=0$ term is the only non-zero term in the sum. The result \eqref{BIO-cond} then implies that the $j=0$ term and the last term cancel for $h\xi\in[-\pi/2,\pi/2]^d$. For $h\xi\notin[-\pi/2,\pi/2]^d$ we have, using Assumption~\ref{def121}(2), that
$(G_0(\xi)+1)^{-1}\leq C h^{\alpha}$ if $0<h\leq h_0$ for $h_0$ sufficiently small.
From the boundedness of $\widehat{\vp}_0$ and $\widehat{\psi}_0$, the norms of the $j=0$ term and the last term are bounded by 
$C h^{\alpha}\norm{g}$.

 For the sum term, the observation \eqref{supp-intersect} implies that we need only consider terms with $\abs{j}=1$. Let $j\in\bZ^d$ with $\abs{j}=1$. Then for $h\xi\in \supp(\widehat{\vp}_0)\cap
 \supp(\widehat{\psi}_0(\cdot-2\pi j))$ we have for some $c_1>0$
\begin{equation*}
\abs{h\xi+2\pi j}\geq c_1,
\end{equation*}
such that
\begin{equation*}
\abs{\xi+\tfrac{2\pi}{h} j}\geq c_1 h^{-1}.
\end{equation*}
This implies for a sufficiently small $h_0$ that we have
\begin{equation}\label{G0-shift}
G_0(\xi+\tfrac{2\pi}{h} j)\geq c h^{-\alpha}
\end{equation}
for $0<h\leq h_0$, by Assumption~\ref{def121}(2).

This estimate together with the boundedness of $\widehat{\vp}_0$ and $\widehat{\psi}_0$ gives
\begin{equation*}
\bigl\vert(2\pi)^d\widehat{\vp}_0(h\xi)
\overline{\widehat{\psi}_0(h\xi+2\pi j)}g(\xi+\tfrac{2\pi}{h}j)(G_0(\xi+\tfrac{2\pi}{h}j)+1)^{-1}
\bigr\vert
\leq C h^{\alpha} \abs{g(\xi+\tfrac{2\pi}{h}j)},
\end{equation*}
for $0<h\leq h_0$ and $\xi\in\bR^d$.
Squaring and integrating over $\bR^d$ gives an estimate of the term for this $j$ with $\abs{j}=1$ of the form $C h^{\alpha}\norm{g}$. By density, after adding up the results for the finite  number of terms with $\abs{j}=1$, the estimate \eqref{eq710} follows.
\end{proof}

Before we can obtain norm resolvent estimates, we need the following result.

\begin{lemma} \label{lemma:diffresG}
	Let $G_0$ satisfy Assumption~{\rm\ref{def121}} and let $\gamma = \min\{2\alpha-1,2\alpha-\beta-1\}$. There exists $C>0$ such that
	\begin{equation}\label{diffG_first}
		\abs{(G_{0,h}(\xi)+1)^{-1}-(G_0(\xi)+1)^{-1}}\leq Ch^{\gamma}, 
	\end{equation}
	for $h>0$ and $h\xi\in[-3\pi/2,3\pi/2]^d$.
\end{lemma}
\begin{proof}
	Let $\eta\in\bR$.
	Taylor's formula with remainder leads to
	\begin{equation*}
		\sin(\eta)=\eta-\frac12\eta^3\int^1_0\sin(t\eta)(1-t)^2\di t.
	\end{equation*}
	Writing $\xi=(\xi_1,\xi_2,\ldots,\xi_d)$ we thus have
	for $i=1,2,\ldots,d$
	\begin{equation}\label{expand}
		\tfrac{2}{h}\sin(\tfrac{h}{2}\xi_i)=
		\xi_i-\frac{h^2}{8}\xi_i^3\int_0^1\sin(\tfrac{h}{2}t\xi_i)
		(1-t)^2\di t.
	\end{equation}
	We also have
	\begin{equation*}
		G_0(\xi+\zeta)=G_0(\xi)+\int_0^1\zeta\cdot\nabla G_0(\xi+t\zeta)dt, \quad \xi,\zeta\in\bR^d.
	\end{equation*}
	Take $\zeta=\zeta_h$, such that $\zeta_{h,i}$ equals the second term on the right hand side of \eqref{expand} for $i=1,2,\ldots,d$, i.e. 
	\begin{equation*} 
		\zeta_{h,i}=-\frac{h^2}{8}\xi_i^3\int_0^1\sin(\tfrac{h}{2}t\xi_i)
		(1-t)^2\di t.
	\end{equation*}
	Thereby
	\begin{equation*}
		G_{0,h}(\xi)-G_0(\xi)=\int_0^1\zeta_h\cdot\nabla G_0(\xi+t\zeta_h)\di t.
	\end{equation*}
	We split the analysis into the cases $\beta\geq0$ and $\beta<0$ in Assumption~\ref{def121}(3).
	\paragraph{The case $\boldsymbol{\beta\geq0}$.}
	Using Assumption~\ref{def121}(3) we get
	\begin{align}
		\abs{G_{0,h}(\xi)-G_0(\xi)}&\leq\abs{\zeta_h}\int_0^1
		\abs{\nabla G_0(\xi+t\zeta_h)} \di t\notag\\
		&\leq C \abs{\zeta_h}\int_0^1
		\abs{\xi+t\zeta_h}^{\beta} \di t\notag\\
		&\leq C \abs{\zeta_h}(\abs{\xi}^{\beta}+\abs{\zeta_h}^{\beta})
		\notag\\
		&\leq C\bigl(
		h^2\abs{\xi}^{3+\beta}+h^{2+2\beta}\abs{\xi}^{3+3\beta}
		\bigr),\label{diff-est}
	\end{align}
	where we used $\abs{\zeta_h}\leq Ch^2\abs{\xi}^3$.

	For $h\xi\in[-3\pi/2,3\pi/2]^d$ we need an estimate independent of $\xi$. It suffices to consider $\abs{\xi}$ sufficiently large. We write
	\begin{equation}\label{diff-expr}
		(G_{0,h}(\xi)+1)^{-1}-(G_0(\xi)+1)^{-1}=
		(G_{0,h}(\xi)+1)^{-1}\bigl(
		G_0(\xi)-G_{0,h}(\xi)
		\bigr)(G_0(\xi)+1)^{-1}.
	\end{equation}
	The middle term on the right hand side is estimated by \eqref{diff-est}. From Assumption~\ref{def121}(2) and for $\abs{\xi}\geq c_0$ we have
	\begin{equation*} 
		(G_0(\xi)+1)^{-1}\leq C \abs{\xi}^{-\alpha}.
	\end{equation*}
For $\abs{\xi}\leq c_0$ we have the estimate $(G_0(\xi)+1)^{-1}\leq1$, since $G_0(\xi)\geq0$. The above estimate then holds also for $\abs{\xi}\leq c_0$ if we take $C$ large enough.
	We note that there exists $c>0$ such that 
	$\abs{\sin(\eta)}\geq c\abs{\eta}$ for $\eta\in[-3\pi/4,3\pi/4]$.
	Thus for
	$h\xi\in[-3\pi/2,3\pi/2]^d$ we have
	\begin{equation*}
		G_{0,h}(\xi)\geq C\Bigl( \sum_{i=1}^d \abs{\tfrac{2}{h}\sin(\tfrac{h}{2}\xi_i)}^2\Bigr)^{\alpha/2}
		\geq C \abs{\xi}^{\alpha},
	\end{equation*}
	such that 
	\begin{equation}\label{res-G0h}
		(G_{0,h}(\xi)+1)^{-1}\leq C \abs{\xi}^{-\alpha}.
	\end{equation}
	Combining \eqref{diff-expr} with these estimates we get 
	\begin{align}\label{est1212}
		\abs{(G_{0,h}(\xi)+1)^{-1}-(G_0(\xi)+1)^{-1}} &\leq C
		\bigl(
		h^2\abs{\xi}^{3+\beta-2\alpha}+h^{2+2\beta}
		\abs{\xi}^{3+3\beta-2\alpha}
		\bigr) \notag\\
		&= C(h\abs{\xi})^{3+\beta-2\alpha}\bigl(1 + (h\abs{\xi})^{2\beta}\bigr)h^{2\alpha-\beta-1}
	\end{align}
	for $h\xi\in[-3\pi/2,3\pi/2]^d$. 
Assumption~\ref{def121} implies $3+\beta-2\alpha \geq 0$, such that the terms in \eqref{est1212} involving $h\abs{\xi}$ are uniformly bounded for $h\xi\in[-3\pi/2,3\pi/2]^d$. 

Hence we get an estimate of the type $Ch^{2\alpha-\beta-1}$ where we also note from Assumption~\ref{def121} that $2\alpha-\beta-1>0$.
	
	\paragraph{The case $\boldsymbol{\beta<0}$.}  Going back to the estimates leading to \eqref{diff-est} we now get
	\begin{equation}
		\abs{G_{0,h}(\xi)-G_0(\xi)}\leq\abs{\zeta_h}\int_0^1
		\abs{\nabla G_0(\xi+t\zeta_h)} \di t
		\leq C \abs{\zeta_h}\leq C h^2\abs{\xi}^3.
		\label{diff-est2}
	\end{equation}
	Repeating the computations leading to \eqref{est1212} yields
	\begin{equation*} 
		\abs{(G_{0,h}(\xi)+1)^{-1}-(G_0(\xi)+1)^{-1}}\leq C
		h^2\abs{\xi}^{3-2\alpha} = C(h\abs{\xi})^{3-2\alpha}h^{2\alpha-1}
	\end{equation*}
	for $h\xi\in[-3\pi/2,3\pi/2]^d$. From Assumption~\ref{def121} we have $3-2\alpha \geq -\beta > 0$, and moreover we have $2\alpha - 1>0$, so we get an estimate of the type $Ch^{2\alpha-1}$. 
\end{proof}

Next we prove the following result. 
\begin{proposition}\label{prop-77}
Let $J_h$ and $K_h$ be defined using biorthogonal Riesz sequences satisfying Assumption~{\rm\ref{assume-72}} and let $G_0$ satisfy Assumption~{\rm\ref{def121}} with parameters $\alpha$ and $\beta$. 
Let the operators $H_0$ and $H_{0,h}$ be defined by \eqref{def-H0} and \eqref{def-H0h}, respectively.
Let 
\begin{equation*} 
\gamma=\min\{2\alpha-1,2\alpha-\beta-1\}.
\end{equation*}
For each compact set $K\subset\bC\setminus[0,\infty)$,
there exist $C>0$ and $h_0>0$ such that
\begin{equation}\label{new-H0-est}
\norm{J_h(H_{0,h}-z\identh)^{-1}K_h-
(H_0-z\ident)^{-1}}_{\cB(\cH)}\leq C h^{\gamma}
\end{equation}
for all $z\in K$ and all $0<h\leq h_0$.
\end{proposition}

\begin{remark} \label{remark:maxorder}
	The largest possible value of $\gamma$ in Proposition~\ref{prop-77} is 2, which is attained precisely when $\beta\geq 1$ and $2\alpha = 3+\beta$.
\end{remark}

\begin{proof}
We start by proving the result for $K=\{-1\}$.
We have
\begin{align*}
J_h(H_{0,h}+&\identh)^{-1}K_h-
(H_0+\ident)^{-1}= \notag\\
&J_h(H_{0,h}+\identh)^{-1}K_h-
J_hK_h(H_0+\ident)^{-1}+(\ident-J_hK_h)(H_0+\ident)^{-1}. \label{propest}
\end{align*}
The term $(\ident-J_hK_h)(H_0+\ident)^{-1}$ is estimated in Lemma~\ref{lemma-74}.

Note that $\gamma\leq \alpha$ for 
the admissible values of $\alpha$ and $\beta$.
The result \eqref{new-H0-est} follows if we for some $h_0>0$ can establish the estimate
\begin{equation} \label{prop34est2}
\norm{J_h(H_{0,h}+\identh)^{-1}K_h-J_hK_h(H_0+\ident)^{-1}}_{\cB(\cH)} \leq Ch^\gamma, \quad 0<h\leq h_0.
\end{equation}
For this purpose, let $f=\cF^{\ast}g$ for $g\in\cS(\bR^d)$. Then
\begin{align}
\cF\bigl(J_h(H_{0,h}+\identh)^{-1}K_h-
J_hK_h(H_0+\ident)^{-1}\bigr)f&=
\cF J_h\sfF_h^{\ast}\sfF_h(H_{0,h}+\identh)^{-1}\sfF_h^{\ast}\sfF_h
K_h\cF^{\ast}g\notag\\
&\quad-\cF J_hK_h\cF^{\ast}\cF(H_0+\ident)^{-1}\cF^{\ast} g.
\label{rewrite1}
\end{align}
A modification of the computations leading to \eqref{FJK} yields
\begin{align}
\bigl[\cF J_h\sfF_h^{\ast}\sfF_h
(H_{0,h}&+\identh)^{-1}\sfF_h^{\ast}\sfF_h
K_h\cF^{\ast}g\bigr](\xi)=\notag\\
&(2\pi)^d\widehat{\vp}_0(h\xi)
(G_{0,h}(\xi)+1)^{-1}\sum_{j\in\bZ^d}
\overline{\widehat{\psi}_0(h\xi+2\pi j)}
g(\xi+\tfrac{2\pi}{h}j).\label{term1}
\end{align}
Analogously
\begin{align}
\bigl[\cF J_hK_h\cF^{\ast}&\cF(H_0+\ident)^{-1}\cF^{\ast}g\bigr](\xi)
=\notag\\
&(2\pi)^d\widehat{\vp}_0(h\xi)
\sum_{j\in\bZ^d}
\overline{\widehat{\psi}_0(h\xi+2\pi j)}
(G_0(\xi+\tfrac{2\pi}{h}j)+1)^{-1}
g(\xi+\tfrac{2\pi}{h}j).\label{term2}
\end{align}
Define  
\begin{equation}\label{def-S}
S_h(\xi)=(G_{0,h}(\xi)+1)^{-1}-
(G_0(\xi)+1)^{-1}.
\end{equation}
Note that
\begin{equation} \label{G0_per}
G_{0,h}(\xi+\tfrac{2\pi}{h}j)=G_{0,h}(\xi),\quad j\in\bZ^d.
\end{equation}
Inserting \eqref{term1} and \eqref{term2} into \eqref{rewrite1}, using the notation \eqref{def-S}, we get
\begin{equation}
\cF\bigl(J_h(H_{0,h}+\identh)^{-1}K_h-
J_hK_h(H_0+\ident)^{-1}\bigr)\cF^{\ast}g= \sum_{j\in\bZ^d} q_{j,h} \label{rewrite2}
\end{equation}
with
\begin{equation*}
	q_{j,h}(\xi) = (2\pi)^d\widehat{\vp}_0(h\xi)
	\overline{\widehat{\psi}_0(h\xi+2\pi j)}
	S_h(\xi+\tfrac{2\pi}{h}j)
	g(\xi+\tfrac{2\pi}{h}j), \quad \xi\in\bR^d, \quad j\in\bZ^d.
\end{equation*}
To estimate the norm of the right hand side of \eqref{rewrite2}, note that as in the proof of Lemma~\ref{lemma-74}, only the terms with $\abs{j}\leq 1$ contribute. First consider the $j=0$ term. We have $\supp \widehat\vp_0\subseteq [-3\pi/2,3\pi/2]^d$ by assumption. From Lemma~\ref{lemma:diffresG} we have $\abs{S_h(\xi)}\leq Ch^{\gamma}$ for $h\xi\in [-3\pi/2,3\pi/2]^d$, which gives
\begin{equation*}
	\norm{q_{0,h}}\leq Ch^{\gamma}\norm{g},
\end{equation*}
since $\widehat{\vp}_0$ and $\widehat{\psi}_0$ are bounded by assumption.

Next let $j\in\bZ^d$ with $\abs{j}=1$ be fixed. Define
\begin{equation*}
	M=((\supp\widehat{\psi}_0)-2\pi j)\cap \supp\widehat{\vp}_0.
\end{equation*}
By assumption we have
\begin{equation*}
	M\subseteq\{\zeta\in\bR^d : {\pi}/{4}\leq\abs{\zeta}
	\leq 3\pi/2\}.
\end{equation*}
Now let $h_0$ be sufficiently small, let $0 < h \leq h_0$, and let $h\xi\in M$. Then similar to \eqref{G0-shift} we have
\begin{equation*}
	(G_0(\xi + \tfrac{2\pi}{h}j) + 1)^{-1} \leq Ch^\alpha,
\end{equation*}
and from \eqref{G0_per} and \eqref{res-G0h} we have
\begin{equation*}
	(G_{0,h}(\xi + \tfrac{2\pi}{h}j) + 1)^{-1} \leq Ch^\alpha.
\end{equation*}
Estimating the two terms in $S_h(\xi+\tfrac{2\pi}{h}j)$ separately leads to
$\abs{S_h(\xi+\tfrac{2\pi}{h}j)}\leq C h^{\alpha}$ for
$h\xi\in M$ and $0<h\leq h_0$.
It follows that
\begin{equation*}
	\norm{q_{j,h}} \leq Ch^{\alpha}\norm{g},\quad 0<h\leq h_0.
\end{equation*}
Since the number of $j\in\bZ^d$ with $\abs{j}=1$ is finite, the result \eqref{prop34est2} follows by density and $\gamma\leq \alpha$.

To finalize the proof, consider a general compact set
$K\subset\bC\setminus[0,\infty)$. To this end note that
\begin{equation*}
	\abs{(G_0(\xi)-z)(G_0(\xi)+1)^{-1}}\leq C
	\quad\text{and}\quad
	\abs{(G_{0,h}(\xi)-z)(G_{0,h}(\xi)+1)^{-1}}\leq C
\end{equation*}
with $C$ independent of $z\in K$, $\xi\in\bR^d$, and $0<h\leq1$.
Thus the crucial estimates of 
\begin{equation*}
	(G_0(\xi)-z)^{-1}-(G_{0,h}(\xi)-z)^{-1}
\end{equation*}
leading to \eqref{diffG_first}, hold in the general case of $z\in K$. Further details are omitted. 
\end{proof} 

To include the fractional Laplacian $(-\Delta)^{\alpha/2}$ for $0<\alpha\leq1$ we have to modify some of the arguments. We start with the following estimate.
\begin{lemma}
Let $0<\alpha\leq1$. Then for $\xi,\zeta\in\bR^d$ we have
\begin{equation}\label{multi-d}
\bigl\lvert\abs{\xi+\zeta}^{\alpha}-\abs{\xi}^{\alpha}
\bigr\rvert\leq\abs{\zeta}^{\alpha}.
\end{equation}
\end{lemma}
\begin{proof}
This result follows from straightforward computations using the inequality $(a+b)^{\alpha}\leq
a^{\alpha}+b^{\alpha}$ 
for $a,b\geq0$ and $0\leq\alpha\leq1$.
\end{proof}

To include the non-differentiable term $\abs{\xi}^{\alpha}$ for $\tfrac12 < \alpha\leq1$ we introduce the following assumption.
\begin{assumption}\label{assume-abs}
	Assume 
	\begin{equation*}
	\tfrac12<\alpha\leq1,\quad \tilde{\beta}\geq0,\quad\text{and}\quad
	1+\tilde{\beta}<2\alpha.
	\end{equation*}
	Let $G_0(\xi)=\abs{\xi}^{\alpha}
	+\widetilde{G}_0(\xi)$, where 
	$\widetilde{G}_0\in C^1(\bR^d)$ is a real-valued and non-negative function. Assume 
\begin{enumerate}[\ \ (1)] 
\item $\widetilde{G}_0(0)=0$,
\item there exist $c>0$ and $c_0>0$ such that $\abs{\nabla \widetilde{G}_0(\xi)}\leq c \abs{\xi}^{\tilde{\beta}}$, for all $\xi\in\bR^d$ with $\abs{\xi}\geq c_0$,
\item $\widetilde{G}_0(\xi_1,\xi_2,\ldots,\xi_d) = \widetilde{G}_0(\abs{\xi_1},\abs{\xi_2},\ldots,\abs{\xi_d})$, for all $\xi\in\bR^d$.
\end{enumerate}
\end{assumption}

\begin{lemma}\label{new-frac}
Let $\vp_0$ and $\psi_0$ satisfy Assumption~{\rm\ref{assume-72}}. 
Let $G_0$ satisfy Assumption~{\rm\ref{assume-abs}}, also allowing $0<\alpha\leq1$, and let $H_0$ be defined by \eqref{def-H0}.
Then there exist $C>0$ and $h_0>0$ such that
\begin{equation*} 
\norm{(J_hK_h-\ident)(H_0+\ident)^{-1}}_{\cB(\cH)}\leq C h^{\alpha},
\quad 0<h\leq h_0.
\end{equation*}
\end{lemma}
\begin{proof}
The proof of Lemma~\ref{lemma-74} can be modified to prove this result. We observe that 
\begin{equation*}
(G_0(\xi)+1)^{-1}\leq(\abs{\xi}^{\alpha}+1)^{-1}\leq C h^{\alpha}
\end{equation*}
for $h\xi\notin[-3\pi/2,3\pi/2]^d$. In the same manner one shows that the estimate \eqref{G0-shift} holds. Further details are omitted. 
\end{proof}

\begin{proposition}\label{prop126}
	Let $J_h$ and $K_h$ be defined using biorthogonal Riesz sequences satisfying Assumption~{\rm\ref{assume-72}} and let $G_0$ satisfy Assumption~{\rm\ref{assume-abs}} with parameters $\alpha$ and $\tilde{\beta}$. Let the operators $H_0$ and $H_{0,h}$ be defined by \eqref{def-H0} and \eqref{def-H0h}, respectively. For each compact set $K\subset\bC\setminus[0,\infty)$,
	there exist $C>0$ and $h_0>0$ such that
	\begin{equation*} 
	\norm{J_h(H_{0,h}-z\identh)^{-1}K_h-
		(H_0-z\ident)^{-1}}_{\cB(\cH)}\leq C h^{2\alpha-\tilde{\beta}-1}
	\end{equation*}
	for all $z\in K$ and all $0<h\leq h_0$.
\end{proposition} 
\begin{proof}
Write $G_{\alpha}(\xi)=\abs{\xi}^{\alpha}$. We now modify the arguments in the proof of Lemma~\ref{lemma:diffresG}.
Let $G_{0,h}$ be defined by \eqref{def-G0h} and  use analogous definitions for $G_{\alpha,h}$ and $\widetilde{G}_{0,h}$. Write
$G_{0,h}(\xi)=G_{\alpha,h}(\xi)+\widetilde{G}_{0,h}(\xi)$, then we have
\begin{equation*}
G_{0,h}(\xi)-G_{0}(\xi)=G_{\alpha,h}(\xi)-G_{\alpha}(\xi)
+\widetilde{G}_{0,h}(\xi)-\widetilde{G}_{0}(\xi).
\end{equation*}
It follows from \eqref{expand}, \eqref{multi-d}, and $\abs{\zeta_h}\leq C h^2\abs{\xi}^3$ that we have
\begin{equation*}
\abs{G_{\alpha,h}(\xi)-G_{\alpha}(\xi)}
\leq C h^{2\alpha}\abs{\xi}^{3\alpha}.
\end{equation*}
Since condition (2) in Assumption~\ref{def121} is not used 
in the proof of the estimates \eqref{diff-est} 
and \eqref{diff-est2}, we have for $\tilde{\beta}\geq0$ the estimate
\begin{equation*}
\abs{\widetilde{G}_{0,h}(\xi)-\widetilde{G}_{0}(\xi)}
\leq C\bigl(
h^2\abs{\xi}^{3+\tilde{\beta}}+h^{2+2\tilde{\beta}}\abs{\xi}^{3+3\beta}
\bigr).
\end{equation*}

We also need two resolvent estimates. Since $\widetilde{G}_0(\xi)\geq0$ we have
\begin{equation*}
(G_0(\xi)+1)^{-1}\leq 
(G_{\alpha}(\xi)+1)^{-1}\leq\abs{\xi}^{-\alpha}
\end{equation*}
and for $h\xi\in[-3\pi/2,3\pi/2]^d$
\begin{equation*}
(G_{0,h}(\xi)+1)^{-1}\leq 
(G_{\alpha,h}(\xi)+1)^{-1}\leq C\abs{\xi}^{-\alpha}
\end{equation*}
by the estimates leading to \eqref{diff-est2}.

Combining the estimates obtained above, we have for $h\xi\in[-3\pi/2,3\pi/2]^d$ the following result: 
\begin{equation}
\abs{(G_{0,h}(\xi)+1)^{-1}-(G_{0}(\xi)+1)^{-1}}\leq
C\bigl(h^{2\alpha}\abs{\xi}^{\alpha}
 +h^2\abs{\xi}^{3+\tilde{\beta}-2\alpha}
 +h^{2+2\tilde{\beta}}\abs{\xi}^{3+3\tilde{\beta}-2\alpha}\bigr).
\label{est-tilde}
\end{equation}

To finish, one can repeat the arguments in the proofs of Lemma~\ref{lemma:diffresG} and Proposition~\ref{prop-77} (with Lemma~\ref{new-frac} in place of Lemma~\ref{lemma-74}) to obtain estimates of type $Ch^\alpha$ and $Ch^{2\alpha-\tilde{\beta}-1}$. Note that for the admissible
values of $\alpha$ and $\tilde{\beta}$ we have 
\begin{equation*}
2\alpha-\tilde{\beta}-1=\min\{\alpha, 2\alpha-\tilde{\beta}-1\}. \qedhere
\end{equation*}
\end{proof}

The following proposition includes the fractional Laplacian $(-\Delta)^{\alpha/2}$ for $0 < \alpha \leq \tfrac12$, which is not part of Propositions~\ref{prop-77} and \ref{prop126}. Moreover, it gives the better rate $h^\alpha$ for the fractional Laplacian for $\tfrac12 < \alpha\leq 1$, compared to the rate $h^{2\alpha-1}$ from Proposition~\ref{prop126}.

\begin{proposition}\label{prop410}	
	Let $J_h$ and $K_h$ be defined using biorthogonal Riesz sequences satisfying Assumption~{\rm\ref{assume-72}} and let $G_0(\xi)=\abs{\xi}^{\alpha}$ for $0<\alpha\leq 1$. Let the operators $H_0$ and $H_{0,h}$ be defined by \eqref{def-H0} and \eqref{def-H0h}, respectively. For each compact set $K\subset\bC\setminus[0,\infty)$, there exist $C>0$ and $h_0>0$ such that
	\begin{equation*} 
	\norm{J_h(H_{0,h}-z\identh)^{-1}K_h-
		(H_0-z\ident)^{-1}}_{\cB(\cH)}\leq C h^{\alpha}
	\end{equation*}
	for all $z\in K$ and all $0<h\leq h_0$.
\end{proposition} 
\begin{proof} 
The proof is a minor modification of the proof of Proposition~\ref{prop126}. In that proof take $\widetilde{G}_0\equiv0$. Then the last two terms on the right hand side of \eqref{est-tilde} vanish, such that
\begin{equation*}
\abs{(G_{0,h}(\xi)+1)^{-1}-(G_{0}(\xi)+1)^{-1}}\leq
C h^{2\alpha}\abs{\xi}^{\alpha}.
\end{equation*}
This leads to 
\begin{equation*}
\abs{(G_{0,h}(\xi)+1)^{-1}-(G_{0}(\xi)+1)^{-1}}\leq C h^{\alpha}
\end{equation*}
for $h\xi\in[-3\pi/2,3\pi/2]^d$, and the result follows as in the proof of Proposition~\ref{prop-77} (with Lemma~\ref{new-frac} in place of Lemma~\ref{lemma-74}).
\end{proof}
\section{Adding a potential}\label{sec:withpotential}
Let $G_0$ be a multiplier satisfying the assumptions in Proposition~\ref{prop-77}, Proposition~\ref{prop126}, or Proposition~\ref{prop410}, and define $H_0$ and $H_{0,h}$ by \eqref{def-H0} and \eqref{def-H0h}, respectively.
We will add a potential to both $H_0$ and $H_{0,h}$ and obtain a norm resolvent convergence result. We use a fairly strong assumption on the potential. 
\begin{assumption}\label{assume-V}
Assume that $V\colon \bR^d\to\bR$ is bounded and H\"{o}lder continuous with exponent $\theta\in(0,1]$.
\end{assumption}

Let $\vp_0$ and $\psi_0$ satisfy Assumption~\ref{assume-72}.
We need an additional assumption on $\psi_0$.
\begin{assumption}\label{assume-beta}
Assume there exists $\tau>d$ such that
\begin{equation*}
\abs{\psi_0(x)}\leq C(1+\abs{x})^{-\tau}\quad\text{for all 
$x\in\bR^d$}.
\end{equation*}
\end{assumption}

We define the discretized potential as 
\begin{equation}
	V_h(k)=V(hk), \quad k\in\bZ^d. \label{discV}
\end{equation} 
Then $H=H_0+V$ is self-adjoint in $\cH$ with domain 
\begin{equation*}
\cD(H)=\cD(H_0)=\{f\in\cH: G_0\widehat{f}\in 
\cH\},
\end{equation*}
and $H_{h}=H_{0,h}+V_h$ is self-adjoint and bounded on $\cH_h$.

The following result is a variant of \cite[Lemma 2.6]{NT}.
\begin{proposition}\label{est-V}
Let $V$ satisfy Assumption~{\rm\ref{assume-V}} 
for a $\theta\in(0,1]$ 
and let $\psi_0$ satisfy Assumption~{\rm\ref{assume-beta}}.  
Then we have
\begin{equation*} 
\norm{V_hK_h-K_hV}_{\cB(\cH,\cH_h)}\leq Ch^{\theta'},
\end{equation*}
where $\theta'$ is given by 
\begin{equation}\label{theta'}
\frac{1}{\theta'}=\frac{1}{\theta}+\frac{1}{\tau-d}.
\end{equation}
\end{proposition} 
\begin{proof} 
Let $f\in \cH$. Then
\begin{equation*}
(V_hK_hf)(k)-(K_hVf)(k)=\int_{\bR^d}K(x,k,h)f(x)\di x, \quad k\in\bZ^d,
\end{equation*}
where
\begin{equation*}
K(x,k,h)=h^{-d}\bigl(V(hk)-V(x)\bigr)\overline{\psi_0((x-hk)/h)}.
\end{equation*}
Schur's test gives boundedness and a bound:
\begin{equation*}
\norm{V_hK_h-K_hV}_{\cB(\cH,\cH_h)}\leq\sqrt{K_1K_2},
\end{equation*}
where
\begin{equation*}
K_1=\sup_{k\in\bZ^d}\int_{\bR^d}\abs{K(x,k,h)}\di x
\end{equation*}
and
\begin{equation*}
K_2=\esup_{x\in\bR^d}\sum_{k\in\bZ^d}\abs{K(x,k,h)}.
\end{equation*}

We first estimate $K_1$. Let $\delta>0$. Then
\begin{equation*}
\int_{\bR^d}\abs{K(x,k,h)}\di x=\int_{\abs{x-hk}<\delta}\abs{K(x,k,h)}\di x
+\int_{\abs{x-hk}\geq\delta}\abs{K(x,k,h)}\di x.
\end{equation*}
Using Assumptions~\ref{assume-V} and~\ref{assume-beta} we get
\begin{equation*}
\int_{\abs{x-hk}<\delta}\abs{K(x,k,h)}\di x
\leq C\delta^{\theta}\int_{\abs{y}<\delta}h^{-d}\abs{\psi_0(y/h)}\di y
\leq C\delta^{\theta}\norm{\psi_0}_{L^1(\bR^d)}.
\end{equation*}
Next we have 
\begin{equation*}
\int_{\abs{x-hk}\geq\delta}\abs{K(x,k,h)}\di x\leq 
C\int_{\abs{h\eta}\geq\delta}\abs{\psi_0(\eta)}\di\eta\leq C (h/\delta)^{\tau-d},
\end{equation*}
where we again used Assumption~\ref{assume-beta}.
Thus we have an estimate
\begin{equation*}
K_1\leq C(\delta^{\theta}+(h/\delta)^{\tau-d}).
\end{equation*}
Let $\gamma\in(0,1)$ and take $\delta=h^{\gamma}$. Then to optimize we solve $\gamma\theta=(1-\gamma)(\tau-d)$ with respect to $\gamma$ to get
\begin{equation*}
\gamma=\frac{\tau-d}{\theta+\tau-d},
\end{equation*}
which gives \eqref{theta'} with $\theta'=\gamma\theta$.
The computations giving the estimate of $K_2$ are very similar to the ones for $K_1$ and give the same result. We omit the details.
\end{proof}

We now combine the 
results above 
to obtain the following main theorem.

\begin{theorem}\label{thm54}
Let $J_h$ and $K_h$ be defined using biorthogonal Riesz sequences satisfying Assumption~{\rm\ref{assume-72}}. Let the potential $V$ satisfy Assumption~{\rm\ref{assume-V}} and define $V_h$ by \eqref{discV}. If $V\not\equiv 0$, let $\psi_0$ satisfy
Assumption~{\rm\ref{assume-beta}} and let $\theta'$ be given by \eqref{theta'}. Let the multiplier $G_0$ satisfy one of three options:
\begin{enumerate}[\rm(i)]
	\item $G_0$ satisfies Assumption~{\rm\ref{def121}} with parameters $\alpha$ and $\beta$.
	\item $G_0$ satisfies Assumption~{\rm\ref{assume-abs}} with parameters $\alpha$ and $\tilde{\beta}$.
	\item $G_0(\xi)=\abs{\xi}^{\alpha}$ for $0<\alpha\leq 1$ and $\xi\in\bR^d$.
\end{enumerate}
Let 
\begin{equation} \label{new-gamma1}
	\gamma = \begin{cases}
		\min\{2\alpha-1,2\alpha-\beta-1,\theta'\} & \text{ for $G_0$ in {\rm(i)},} \\
		\min\{2\alpha-\tilde{\beta}-1,\theta'\} & \text{ for $G_0$ in {\rm(ii)},} \\
		\min\{\alpha,\theta'\} & \text{ for $G_0$ in {\rm(iii)}.} 
	\end{cases}
\end{equation}
Let $H=H_0+V$ and $H_h=H_{0,h}+V_h$ where the operators $H_0$ and $H_{0,h}$ are defined by \eqref{def-H0} and \eqref{def-H0h}, respectively. For each compact set $K\subset\bC\setminus\bR$, there exist $C>0$ and $h_0>0$ such that
\begin{equation}\label{thm1}
\norm{J_h(H_{h}-z\identh)^{-1}K_h-
(H-z\ident)^{-1}}_{\cB(\cH)}\leq C h^{\gamma}
\end{equation}
for all $z\in K$ and all $0<h\leq h_0$.
\end{theorem}
\begin{proof}

We give a proof for $G_0$ in (i). For the other two classes (ii) and (iii) of $G_0$, the proof is analogous but makes use of Proposition~\ref{prop126} and Proposition~\ref{prop410}, respectively, instead of Proposition~\ref{prop-77}; we omit the details for (ii) and (iii).	

Let $z\in K$. Then
\begin{align*}
J_h(H_{h}-z\identh)^{-1}&K_h-
(H-z\ident)^{-1}=\\
&J_h(H_{h}-z\identh)^{-1}K_h-
J_hK_h(H-z\ident)^{-1}+(\ident-J_hK_h)(H-z\ident)^{-1}.
\end{align*}
We have
\begin{equation}
(\ident-J_hK_h)(H-z\ident)^{-1}=(\ident-J_hK_h)(H_0-z\ident)^{-1}(H_0-z\ident)(H-z\ident)^{-1}. \label{thmterm1}
\end{equation}
Since the norm of $(H_0-z\ident)(H-z\ident)^{-1}$ is bounded uniformly in $z\in K$, Lemma~\ref{lemma-74} implies that
\begin{equation*}
\norm{(\ident-J_hK_h)(H-z\ident)^{-1}}_{\cB(\cH)}\leq C h^{\alpha}.
\end{equation*}
Next we have
\begin{equation*}
J_h(H_{h}-z\identh)^{-1}K_h-
J_hK_h(H-z\ident)^{-1}=J_h\bigl((H_h-z\identh)^{-1}K_h-
K_h(H-z\ident)^{-1}\bigr).
\end{equation*}
Since $\norm{J_h}_{\cB(\cH_h,\cH)}$ is bounded uniformly in $h$ by \eqref{J-norm}, we rewrite the second factor as
\begin{align}
(H_{h}-z\identh)^{-1}K_h-K_h(H-z\ident)^{-1}&=
\bigl[(H_h-z\identh)^{-1}(H_{0,h}-z\identh)\bigr] \notag\\
\cdot\bigl((H_{0,h}-z\identh)^{-1}(K_hH&-H_hK_h)(H_0-z\ident)^{-1}\bigr)
\bigl[(H_0-z\ident)(H-z\ident)^{-1}\bigr]. \label{thmterm2}
\end{align}
The terms in square brackets are bounded in norm, uniformly in $z\in K$ and $h$. The middle term is split in two.
\begin{align*}
(H_{0,h}-z\identh)^{-1}(K_hH-H_hK_h)(H_0-z\ident)^{-1}&=
(H_{0,h}-z\identh)^{-1}(K_hH_0-H_{0,h}K_h)(H_0-z\ident)^{-1}\\
&\quad+(H_{0,h}-z\identh)^{-1}(K_hV-V_hK_h)(H_0-z\ident)^{-1}.
\end{align*}
The first term on the right hand side is written as
\begin{align*}
(H_{0,h}-z\identh)^{-1}(K_hH_0-H_{0,h}K_h)(H_0-z\ident)^{-1}
&=(H_{0,h}-z\identh)^{-1}K_h-K_h(H_0-z\ident)^{-1}\\
&=
K_h\bigl(
J_h(H_{0,h}-z\identh)^{-1}K_h-(H_0-z\ident)^{-1}
\bigr),
\end{align*}
where we used $K_hJ_h=\identh$, see \eqref{KhJh}. Since $\norm{K_h}_{\cB(\cH,\cH_h)}$ is uniformly bounded in $h$ by \eqref{K-norm}, we can use 
\eqref{new-H0-est} in Proposition~\ref{prop-77} to estimate this term.

The term 
\begin{equation*}
(H_{0,h}-z\identh)^{-1}(K_hV-V_hK_h)(H_0-z\ident)^{-1}
\end{equation*}
is estimated using Proposition~\ref{est-V}, with the norms of the resolvents on either side uniformly bounded in $z\in K$ and $h$. Combining these estimates gives the result \eqref{thm1} with $\gamma$ from \eqref{new-gamma1}.
\end{proof}

We get the following additional estimates based on Theorem~\ref{thm54} and the Helffer--Sj\"ostrand formula \cite[Proposition~7.2]{HS}.

\begin{corollary}\label{corspec}
	Let $H$, $H_h$, and $\gamma$ satisfy the assumptions in Theorem~{\rm\ref{thm54}}. There exists $N\in\bN$, so for each compact interval $[a,b]$, there are $C>0$ and $h_0>0$ such that
	\begin{equation}\label{hc7}
	\norm{J_h (H_{h}-z\identh)^{-1} K_h -(H-z\ident)^{-1}}_{\cB(\cH)} \leq C \abs{y}^{-N} h^\gamma
	\end{equation}
	for all $0<h\leq h_0$ and $z=x+\I y$, with $x\in[a,b]$ and real $y$ satisfying $0<|y|\leq 1$. 
	
	Now let $F\in C_0^\infty(\bR)$ with $\supp F\subseteq [a,b]$. Then 
	\begin{equation}\label{hc2}
	\norm{J_h F(H_{h})K_h-F(H)}_{\cB(\cH)}
	\leq C h^{\gamma},\quad 0<h\leq h_0.
	\end{equation}
\end{corollary}
\begin{proof}
	Since $H_{0,h}$ and $H_0$ are self-adjoint, the norms $\norm{(H_{0,h}-z\identh)^{-1}(H_{0,h}+\I\identh)}_{\cB(\cH_h)}$ and $\norm{(H_0+\I\ident)(H_0-z\ident)^{-1}}_{\cB(\cH)}$ are bounded by $C\abs{y}^{-1}$ uniformly in $h$.
	
	For $V\equiv 0$, the norm of \eqref{thmterm1} is estimated by
	\begin{equation*}
		C\abs{y}^{-1}\norm{(\ident-J_hK_h)(H_0+\I\ident)^{-1}}_{\cB(\cH)}
	\end{equation*}
	and the norm of the right hand side of \eqref{thmterm2} is estimated by
	\begin{equation*}
		C\abs{y}^{-2}\norm{(H_{0,h}+\I\identh)^{-1}(K_hH_0-H_{0,h}K_h)(H_0+\I\ident)^{-1}}_{\cB(\cH,\cH_h)}.
	\end{equation*}
	Continuing as in the proof of Theorem~\ref{thm54} gives the result \eqref{hc7} with $N = 2$.
	
	For $V\not\equiv 0$, again considering \eqref{thmterm1} and \eqref{thmterm2}, we get the same result up to estimates of $\norm{(H_h-z\identh)^{-1}(H_{0,h}-z\identh)}_{\cB(\cH_h)}$ and $\norm{(H_0-z\ident)(H-z\ident)^{-1}}_{\cB(\cH)}$. As $V$ is bounded, these norms are both bounded by $C\abs{y}^{-1}$ uniformly in $h$. Thereby \eqref{hc7} holds for $N = 4$.	
	
	Now let us prove \eqref{hc2}. We can find an almost analytic extension  $\widetilde{F}(z)$ with support in $\{x+\I y : a\leq x\leq b,\; -1\leq y \leq 1\}$ such that $\partial_{\overline{z}} \widetilde{F}(z)$ behaves like $|y|^N$ when $|y|$ is small. The last ingredient is the Helffer--Sj\"ostrand formula \cite[Proposition~7.2]{HS}. 
\end{proof}

\section{Spectral estimates} \label{sec:spectral}

One can apply Theorem~\ref{thm54} and Corollary~\ref{corspec} to obtain some spectral information. 

Let $H$ and $H_h$ be as in Theorem~\ref{thm54}. Since $H_0$ and $H_{0,h}$ are non-negative and $V$ is bounded, there are
\begin{equation} \label{mucondition}
	\mu > 0 \quad \text{such that} \quad (-\infty,-\mu] \subseteq \rho(H)\cap \rho(H_h)  \text{ for all } h>0,
\end{equation}
with $\rho(H)$ and $\rho(H_h)$ being the resolvent sets. Before stating estimates on the spectrum of $(H+\mu\ident)^{-1}$ and $(H_h+\mu\identh)^{-1}$, we need the following technical lemma as $J_h(H_h+\mu\identh)^{-1}K_h$ may not be a normal operator. 

We denote the Hausdorff distance between two compact and non-empty sets $X,Y\subset\bC$ by $\haus(X,Y)$.

\begin{lemma}\label{hausestlemma}
	Let $A_1, A_2\in\cB(\cK)$ for a Banach space $\cK$. For $j=1,2$ assume there exists a constant $C>0$ such that for all $z\in \rho(A_j)$, the resolvent set of $A_j$, we have
	\begin{equation*}
		\norm{(A_j-z\ident)^{-1}}_{\cB(\cK)} \leq \frac{C}{\dist(z,\sigma(A_j))}.
	\end{equation*}
	Then $\haus\big (\sigma(A_1),\sigma(A_2)\big )\leq C \norm{A_1-A_2}_{\cB(\cK)}$.
\end{lemma}
\begin{proof}
	By symmetry, it is enough to prove the following statement: if $\lambda\in\sigma(A_1)$, then there exists $\lambda'\in\sigma(A_2)$ such that $|\lambda-\lambda'|\leq C\norm{A_1-A_2}_{\cB(\cK)}$. 
	
	Let us assume the contrary: there exists some $\lambda\in\sigma(A_1)$ such that the closed complex disk 
	\begin{equation*}
		D=\{z\in\bC : \abs{z-\lambda} \leq C\norm{A_1-A_2}_{\cB(\cK)}\}
	\end{equation*}
	is included in the resolvent set of $A_2$. Thus $D\cap \sigma(A_2)=\emptyset$ and in particular
	\begin{equation*}
		\dist(\lambda,\sigma(A_2)) > C\norm{A_1-A_2}_{\cB(\cK)}.
	\end{equation*}
	By writing
	\begin{equation*}
		A_1-\lambda\ident=\big (\ident +(A_1-A_2)(A_2-\lambda\ident)^{-1}\big )(A_2-\lambda\ident)
	\end{equation*}
	and noticing that
	\begin{equation*}
		\norm{(A_1-A_2)(A_2-\lambda\ident)^{-1}}_{\cB(\cK)} \leq \frac{C\norm{A_1-A_2}_{\cB(\cK)}}{\dist(\lambda,\sigma(A_2))} < 1,
	\end{equation*}
	we conclude that $A_1-\lambda\ident$ is invertible; a contradiction.
\end{proof}

\begin{theorem}\label{resolventest}
Let $H$, $H_h$, and $\gamma$ satisfy the assumptions in Theorem~{\rm\ref{thm54}}. Let $\mu>0$ satisfy \eqref{mucondition}. Then there exist $C>0$ and $h_0>0$ such that
\begin{equation*}
\haus\bigl(\sigma((H_h+\mu\identh)^{-1}),\sigma((H+\mu\ident)^{-1})\bigr)
\leq C h^{\gamma},\quad 0<h\leq h_0.
\end{equation*}
\end{theorem}
\begin{proof}
Proposition~\ref{prop-spect} implies
\begin{equation}\label{spect2}
\sigma(J_h(H_h+\mu\identh)^{-1}K_h)
=\sigma((H_h+\mu\identh)^{-1})\cup
\{0\}.
\end{equation}
Let $z$ belong to the resolvent set of $J_h(H_h+\mu\identh)^{-1}K_h$. Let $P_h=J_hK_h$. Using the decomposition $\cH=P_h\cH\oplus(I-P_h)\cH$ we have
\begin{equation*}
	(J_h(H_h+\mu\identh)^{-1}K_h - z\ident)^{-1} = \begin{bmatrix}
	J_h((H_h+\mu\identh)^{-1}-z\identh)^{-1}K_h & 0\\
	0 & -z^{-1}(I-P_h)
	\end{bmatrix}.
\end{equation*}

The norms of $J_h$, $K_h$, and $I-P_h$ are bounded uniformly in $h$. Hence, as $H_h$ and $H$ are self-adjoint, the operators $J_h(H_h+\mu\identh)^{-1}K_h$ and $(H+\mu\ident)^{-1}$ satisfy the assumptions of Lemma~\ref{hausestlemma} with a constant $C$ independent of $h$. Lemma~\ref{hausestlemma} and Theorem~\ref{thm54} give the estimate 
\begin{align*}
\haus\bigl(
\sigma(J_h(H_h+\mu\identh)^{-1}K_h),\sigma((H+\mu\ident)^{-1})
\bigr)&\leq C \norm{J_h(H_h+\mu\identh)^{-1}K_h-(H+\mu\ident)^{-1}}_{\cB(\cH)} \\
&\leq C h^{\gamma}, \quad 0<h\leq h_0.
\end{align*}
	
We have from \eqref{spect2} and the triangle inequality that
\begin{align*}
\haus\bigl(\sigma((H_h+\mu\identh)^{-1}),\sigma((H+\mu\ident)^{-1})\bigr)&\leq
\haus\bigl(\sigma((H_h+\mu\identh)^{-1}),\sigma((H_h+\mu\identh)^{-1})\cup
\{0\}\bigr)\\
&\quad+\haus\bigl(
\sigma(J_h(H_h+\mu\identh)^{-1}K_h),\sigma((H+\mu\ident)^{-1})
\bigr).
\end{align*}
The estimate of the first term on the right hand side depends on the case in Theorem~\ref{thm54}. Consider the case (i).
Assumption~\ref{def121} implies that we have $\abs{G_0(\xi)}
\leq c\abs{\xi}^{1+\beta}$. This estimate yields
\begin{equation*}
\abs{G_{0,h}(\xi)}\leq C h^{-1-\beta}.
\end{equation*}
Since $V$ is bounded we conclude that there exist $C>0$ and $h_0>0$ such that
\begin{equation*}
\max\sigma(H_h)\leq C h^{-1-\beta},\quad 0<h\leq h_0.
\end{equation*}
Then
\begin{align*}
\haus\bigl(\sigma((H_h+\mu\identh)^{-1}),\sigma((H_h+\mu\identh)^{-1})\cup
\{0\}\bigr)&=\dist\bigl(0,\sigma((H_h+\mu\identh)^{-1})\bigr)\leq C h^{1+\beta},
\end{align*}
for all $0<h\leq h_0$. From Assumption~\ref{def121} and Theorem~\ref{thm54} we have $\gamma\leq \alpha\leq 1+\beta$, so the result follows in this case. Similar arguments apply in the other cases of Theorem~\ref{thm54}.
\end{proof}

For estimates on the spectrum of the original operators $H_h$ and $H$, we obtain local results as $\sigma(H)$ is unbounded. 

First we define a local version of a Hausdorff distance. 
\begin{definition}\label{def:localhaus}
	For a compact set $K\subset \bC$, we define the local Hausdorff distance in $K$ between two closed and non-empty sets $X,Y\subseteq\bC$ by 
	\begin{equation*}
	\hausK(X,Y) = \max\{ 0, \sup_{x\in X\cap K}\dist(x,Y), \sup_{y\in Y\cap K}\dist(y,X) \}.
	\end{equation*}
	We use the convention $\sup\emptyset = -\infty$. In particular, the sets $X\cap K$ and $Y\cap K$ are allowed to be empty.
\end{definition}
\begin{remark} 
	The following basic properties hold for $\hausK$: 
	\begin{enumerate}[(i)]
		\item $0\leq \hausK(X,Y) < \infty$,
		\item $\hausK(X,Y) = \hausK(Y,X)$, 
		\item $\hausK(X,Y) \leq \haus(X,Y)$ and with equality for $X,Y\subseteq K$,
		\item $\hausK(X,Y) = 0$ if and only if $X\cap K = Y\cap K$.
	\end{enumerate}
	We note that $\haus(X\cap K,Y\cap K)$ is not the right quantity to describe the local closeness between two sets. For example let $K=[0,1]$, $X=\{0\}\cup [1,\infty)$, and $Y=\{h,1+h\}$ where $0<h<1/2$. Then $\haus(X\cap K,Y\cap K)=1-h$  while $\hausK(X,Y)=h$.
\end{remark}

We now obtain the following local estimates.

\begin{theorem}\label{localest}
	Let $H$, $H_h$, and $\gamma$ satisfy the assumptions in Theorem~{\rm\ref{thm54}}. Let $\mu > 0$ satisfy \eqref{mucondition}, let $-\mu< a\leq b$, and let $K=[a,b]$. Then there exist $C>0$ and $h_0>0$ such that 
	\begin{equation*}
		\hausK\big (\sigma(H_h),\sigma(H)\big )\leq (b+\mu)^2 C h^\gamma,
	\end{equation*}
	for all $0<h\leq h_0$. The constant $C$ does not depend on $a$ and $b$.
\end{theorem}
\begin{proof}
	Without loss of generality, we can assume that $\sigma(H)\cap K$ and $\sigma(H_h)\cap K$ are non-empty. 
	
	Let $x\in \sigma(H)\cap K$. Define $\lambda=1/(x+\mu)\in \sigma((H+\mu\ident)^{-1})$. From Theorem~\ref{resolventest} there exists $\lambda'\in \sigma((H_h+\mu\identh)^{-1})$ such that
	\begin{equation*}
		\abs{\lambda-\lambda'} \leq \haus\bigl(\sigma((H_h+\mu\identh)^{-1}),\sigma((H+\mu\ident)^{-1})\bigr) \leq Ch^\gamma.
	\end{equation*}
	Since $\lambda\geq 1/(b+\mu) > 0$ the above inequality, for $0<h\leq h_0$, implies the existence of $h_0$ small enough that
	\begin{equation*}
		\lambda'\geq \frac{1}{2(b+\mu)}.
	\end{equation*}
	
	Let $x' = 1/\lambda'-\mu\in \sigma(H_h)$. Then
	\begin{equation*}
		\abs{x-x'} = \frac{\abs{\lambda-\lambda'}}{\lambda\lambda'}\leq 2(b+\mu)^2 Ch^\gamma.
	\end{equation*}
	By symmetry, the estimate also holds when the roles of $H$ and $H_h$ are reversed.
\end{proof}  

Based on Theorem~\ref{localest} we immediately obtain the result, that for any gap in $\sigma(H)$ there is likewise a gap in $\sigma(H_h)$ of at least the same size for all sufficiently small $h$.
\begin{corollary}\label{cor_gap}
	Let $H$ and $H_h$ satisfy the assumptions in Theorem~{\rm\ref{thm54}}. Let $a,b\in\bR$ with $a < b$ and $[a,b]\cap \sigma(H)=\emptyset$. Then there exists $h_0>0$ such that 
	\begin{equation}\label{band}
	[a,b]\cap\sigma(H_h)=\emptyset,\quad 0<h\leq h_0.
	\end{equation}
\end{corollary}
\begin{proof}
	Since $\sigma(H)$ is closed there is an $\epsilon>0$ such that $\rho(H)$, the resolvent set of $H$, contains $[a-2\epsilon,b+2\epsilon]$. Let $K = [a-\varepsilon,b+\varepsilon]$. For the purpose of reaching a contradiction, assume there is a sequence $\{h_n\}_{n\in\bN}$, $h_n\to 0$ for $n\to\infty$, satisfying $K\cap \sigma(H_{h_n})\neq \emptyset$. Then the local Hausdorff distance $\hausK\big(\sigma(H),\sigma(H_{h_n})\big)$ is bounded from below by $\epsilon$ when $h_n\to 0$, contradicting Theorem \ref{localest}. Hence there is an $h_0>0$ such that $K\subset \rho(H_h)$ for $0<h\leq h_0$, in particular $[a,b]\cap\sigma(H_h) = \emptyset$.
\end{proof}

Denote the spectral measures of $H_h$ and $H$ by $E_{H_h}$ and $E_H$, respectively. Corollary~\ref{corspec} and Theorem~\ref{localest} imply the following result (cf.~\cite[Corollary~1.2]{NT}). Note that \eqref{proj} could have been proved directly from Theorem~\ref{thm54} by adapting the arguments in the proof of \cite[Theorem VIII.23]{RSI} to the present setting. Here we have chosen to base the proof on Corollary~\ref{corspec} and Theorem~\ref{localest} to show the scope of these results.

\begin{corollary}\label{cor56}
Let $H$, $H_h$, and $\gamma$ satisfy the assumptions in Theorem~{\rm\ref{thm54}}. Let $a,b\in\bR$ with $a < b$ and $a,b\not\in\sigma(H)$. Then there exist $C>0$ and $h_0>0$ such that
\begin{equation}\label{proj}
\norm{J_hE_{H_{h}}((a,b))K_h-E_H((a,b))}_{\cB(\cH)}
\leq C h^{\gamma},\quad 0<h\leq h_0.
\end{equation}
\end{corollary}
\begin{proof}
Similar to the proof of Corollary~\ref{cor_gap}, there is an $\epsilon>0$ such that $\rho(H)$ contains $[a-\varepsilon,a+\varepsilon]\cup[b-\epsilon,b+\epsilon]$, and the same holds for $\rho(H_h)$ for $0<h\leq h_0$ with some small enough $h_0>0$ by Corollary~\ref{cor_gap}. Now choose a smooth cut-off function $F$ as in Corollary~\ref{corspec}, which equals $1$ on $[a-\varepsilon/2,b+\varepsilon/2]$ and $\supp F \subset (a-\varepsilon,b+\varepsilon)$. If $0<h\leq h_0$ we have 
\begin{equation*}
F(H)=E_H((a,b)) \quad \text{and}\quad  F(H_h)=E_{H_h}((a,b)).
\end{equation*}
Thus \eqref{proj} follows from Corollary~\ref{corspec}.
\end{proof}

Using Corollary~\ref{cor56} we can get the following results on the relation between eigenvalues of $H$ and $H_h$.

\begin{theorem}\label{eigen-thm}
Let $H$, $H_h$, and $\gamma$ satisfy the assumptions in Theorem~{\rm\ref{thm54}}. Assume that $\lambda$ is an isolated eigenvalue of $H$ with multiplicity $m<\infty$. Let $a,b\in\bR$, $a<\lambda<b$,  and $[a,b]\cap\sigma(H)=\{\lambda\}$. Then there exist discrete eigenvalues of $H_h$ denoted $\lambda_{i,h}\in\bR$, 
$i=1,2,\ldots,m$, and $h_0>0$ such that
\begin{equation}\label{lambdaih}
(a,b)\cap\sigma(H_h)=\{\lambda_{i,h}:i=1,2,\ldots,m\},\quad
0<h\leq h_0.
\end{equation}
Let $\psi_j\in\cH$, $\Vert \psi_j\Vert=1$, linearly independent, and $H\psi_j=\lambda\psi_j$ for $j=1,2,\ldots,m$. Then there are $C>0$ and $h_0>0$ such that for all $0<h\leq h_0$ there exist $\psi_{j,h}\in E_{H_h}((a,b))\cH_h$, $j=1,2,\ldots,m$, forming a basis for the range of $E_{H_h}((a,b))$ and 
\begin{equation}\label{eigen-est}
\norm{\psi_{j,h}-K_h\psi_j}_{\cH_h}\leq Ch^{\gamma},\quad \norm{K_h\psi_j}_{\cH_h}\geq C^{-1}, \quad j = 1,2,\ldots,m.
\end{equation}
\end{theorem}
\begin{proof}
To simplify the notation let $\cI=(a,b)$. Corollary~\ref{cor56} implies that we can take $h_0$ sufficiently small such that we get 
\begin{equation*}
\norm{J_hE_{H_{h}}(\cI)K_h-E_H(\cI)}_{\cB(\cH)}
<1,\quad 0<h\leq h_0.
\end{equation*}
Note that $J_hE_{H_{h}}(\cI)K_h$ is a projection due to $K_hJ_h=\identh$. Then standard results from perturbation theory, see e.g.~\cite[I--\S4.6]{kato}, imply the existence of an invertible operator 
$U_h\in\cB(\cH)$ such that 
\begin{equation}
U_hE_H(\cI)=J_hE_{H_h}(\cI)K_hU_h, \label{hc30}
\end{equation}
and
\begin{equation}
\norm{U_h-I}_{\cB(\cH)} \leq C \norm{E_H(\cI)-J_hE_{H_h}(\cI)K_h}_{\cB(\cH)} \leq Ch^\gamma. \label{hc302}
\end{equation}
There is an explicit expression for $U_h$ known as the Nagy operator, see~\cite[II--(4.18)]{kato}. By assumption $\rank E_H(\cI)=m$. Using the trace cyclicity we have 
\begin{equation*}
m = \trace(U_hE_H(\cI)U_h^{-1})
=\trace(J_hE_{H_h}(\cI)K_h)=\trace(K_hJ_hE_{H_h}(\cI))
=\trace(E_{H_h}(\cI)),
\end{equation*}
thus $\rank E_{H_h}(\cI)=m$, $0<h \leq h_0$, and \eqref{lambdaih} is proved.

We now construct the basis $\psi_{j,h}$, $j=1,2,\ldots,m$, in \eqref{eigen-est}. Let $\psi_j\in E_H(\cI)\cH$, $j=1,2,\ldots, m$, be an orthonormal basis for the range of $E_H(\cI)$. From \eqref{hc30} we have that $\varphi_{j,h}=U_h\psi_j$ obeys 
\begin{equation}\label{hc31}
\varphi_{j,h}=J_hE_{H_h}(\cI)K_h\varphi_{j,h},\quad K_h\varphi_{j,h}=E_{H_h}(\cI)K_h\varphi_{j,h}.
\end{equation}
Define $\psi_{j,h}=K_h\varphi_{j,h}$, $j=1,2,\ldots,m$. In order to form an $m$-dimensional basis in the range of $E_{H_h}(\cI)$ the $\psi_{j,h}$'s have to be linearly independent. Suppose they are not, then the first equality in \eqref{hc31} implies that the $\varphi_{j,h}$'s are not linearly independent either; a contradiction. 

Finally, let us prove the estimates in \eqref{eigen-est}. First, due to \eqref{K-norm} and \eqref{hc302}, we get that the norm of $\psi_{j,h}-K_h\psi_j$ is of order $h^\gamma$. Also, because $\psi_j$ belongs to the domain of $H_0$ we have 
\begin{align*}
	\norm{(J_hK_h - I)\psi_j} \leq C\norm{(J_hK_h-I)(H_0+I)^{-1}}_{\cB(\cH)} \leq Ch^\gamma, \quad 0<h\leq h_0,
\end{align*}
where we used either Lemma~\ref{lemma-74} or Lemma~\ref{new-frac} (depending on the case in Theorem~\ref{thm54}) and $\gamma\leq\alpha$. Using \eqref{J-norm}, this leads to 
\begin{equation*}
1=\norm{\psi_j} \leq \norm{J_h}_{\cB(\cH_h,\cH)}\norm{K_h\psi_{j}}_{\cH_h} + \mathcal{O}(h^\gamma)\leq C \norm{K_h\psi_{j}}_{\cH_h}+\tfrac{1}{2},\quad 0<h\leq h_0,
\end{equation*}
for a small enough $h_0$, thus concluding the proof. 
\end{proof}

\begin{remark}\label{rmk57} \
	\begin{enumerate}[(i)] 
		\item If $\lambda$ is a simple eigenvalue, i.e.\ $m=1$, we get that $\psi_{1,h}$ is an eigenfunction corresponding to the eigenvalue $\lambda_{1,h}$.
		\item For each $n\geq n_0$ sufficiently large we take $a=\lambda-n^{-1}$ and $b=\lambda+n^{-1}$. This allows us to determine a sequence $\{h_n\}$, $h_n\to0$ as $n\to\infty$, and an eigenvalue $\lambda_{h_n}$  of $H_{h_n}$ such that $\lambda_{h_n}\to\lambda$ as $n\to\infty$. This result was proved in \cite[Theorem 23]{R} for a different class of potentials $V$, which has some overlap with our class.
	\end{enumerate}
\end{remark}

\paragraph*{Acknowledgments.} This research is partially supported by grant 8021--00084B from Independent Research Fund Denmark \textbar\ Natural Sciences.


\end{document}